\documentclass[12pt,a4paper]{article}
\usepackage{amsmath}
\usepackage{amsthm}
\usepackage{amssymb}
\usepackage{amscd}
\usepackage{graphicx}
\usepackage{epsfig}
\usepackage[matrix,arrow,curve]{xy}
\usepackage{color}
\usepackage{mathrsfs}
\usepackage[scr=rsfs,cal=boondox]{mathalpha}

\usepackage{bbm}
\usepackage{stmaryrd}
\usepackage{hyperref}

\usepackage{latexsym}
\usepackage{amsfonts}
\input xy

\xyoption{all} \tolerance=500

\newtheorem{theorem}{Theorem}[section]
\newtheorem{proposition}[theorem]{Proposition}
\newtheorem{lemma}[theorem]{Lemma}
\newtheorem{corollary}[theorem]{Corollary}

\newtheorem{conjecture}[theorem]{Conjecture}

\theoremstyle{definition}
\newtheorem{example}[theorem]{Example}
\newtheorem{remark}[theorem]{Remark}
\newtheorem{definition}[theorem]{Definition}

\mathchardef\za="710B  
\mathchardef\zb="710C  
\mathchardef\zg="710D  
\mathchardef\zd="710E  
\mathchardef\zve="710F 
\mathchardef\zz="7110  
\mathchardef\zh="7111  
\mathchardef\zvy="7112 
\mathchardef\zi="7113  
\mathchardef\zk="7114  
\mathchardef\zl="7115  
\mathchardef\zm="7116  
\mathchardef\zn="7117  
\mathchardef\zx="7118  
\mathchardef\zp="7119  
\mathchardef\zr="711A  
\mathchardef\zs="711B  
\mathchardef\zt="711C  
\mathchardef\zu="711D  
\mathchardef\zvf="711E 
\mathchardef\zq="711F  
\mathchardef\zc="7120  
\mathchardef\zw="7121  
\mathchardef\ze="7122  
\mathchardef\zy="7123  
\mathchardef\zf="7124  
\mathchardef\zvr="7125 
\mathchardef\zvs="7126 
\mathchardef\zf="7127  
\mathchardef\zG="7000  
\mathchardef\zD="7001  
\mathchardef\zY="7002  
\mathchardef\zL="7003  
\mathchardef\zX="7004  
\mathchardef\zP="7005  
\mathchardef\zS="7006  
\mathchardef\zU="7007  
\mathchardef\zF="7008  
\mathchardef\zW="700A  

\newcommand{\be}{\begin{equation}}
\newcommand{\ee}{\end{equation}}

\newcommand{\bea}{\begin{eqnarray}}
\newcommand{\eea}{\end{eqnarray}}
\newcommand{\beas}{\begin{eqnarray*}}
\newcommand{\eeas}{\end{eqnarray*}}
\def\*{{\textstyle *}}
\newcommand{\nn}{\nonumber}

\newcommand{\pa}{\partial}
\newcommand{\ti}{\times}

\newcommand{\Ll}{{\pounds}}

\def\ran{\rangle}

\def\End{\mathsf{End}}

\def\cA{{\cal A}}

\def\cE{{\cal E}}

\def\wt{\widetilde}

\def\sT{{\mathsf T}}
\def\sV{{\mathsf V}}

\def\xd{\mathrm{d}}

\def\dt{\xd_{\mathsf T}}
\def\dts{\xd_{{\mathsf T}^*}}

\def\cM{{\cal M}}

\newdir{|>}{%
!/4.5pt/@{|}*:(1,-.2)@^{>}*:(1,+.2)@_{>}}

\newcommand{\g}{\mathfrak{g}}

\newcommand{\la}{\langle}

\newcommand{\N}{\mathbb{N}}
\newcommand{\Z}{\mathbb{Z}}
\newcommand{\R}{\mathbb{R}}

\newcommand{\C}{\mathbb{C}}

\newcommand{\n}{\nabla}

\def\deg{\mathsf{deg}}

\newcommand{\id}{\on{id}}

\newcommand{\catname}[1]{\textnormal{\texttt{#1}}}

\newcommand{\rmd}{\textnormal{d}}

\newcommand{\pr}{\textnormal{pr}}

\newcommand{\mn}{{\medskip\noindent}}

\newcommand{\no}{{\noindent}}

\def\on{\operatorname}
\def\op{\oplus}

\def\g{\mathfrak{g}}

\newcommand{\w}{{\mathsf w}}

\def\inv{\on{inv}}

\DeclareMathOperator{\GL}{GL}
\DeclareMathOperator{\SL}{SL}

\DeclareMathOperator{\Ker}{Ker}

\DeclareMathOperator{\Exp}{Exp}

\newcommand{\we}{\wedge}
\newcommand{\ber}{\operatorname{Ber}}
\def\fX{\mathfrak{X}}
\def\tM{{\wt M}}

\begin{document}
\title{\bf Graded supermanifolds and homogeneity\footnote{Research funded by the National Science Centre (Poland) within the project WEAVE-UNISONO, No. 2023/05/Y/ST1/00043.}}
\date{}
\author{\\ Katarzyna  Grabowska$^1$\\ Janusz Grabowski$^2$
        \\ \\
         $^1$ {\it Faculty of Physics,\
                University of Warsaw}\\
                $^2$ {\it Institute of Mathematics,\
                Polish Academy of Sciences}
                }
\maketitle

\begin{abstract}
We introduce the concept of a \emph{homogeneity supermanifold}, which is, roughly speaking, a supermanifold equipped with a privileged atlas whose coordinates carry prescribed (real) homogeneity degrees. This structure defines a sheaf of graded algebras on the supermanifold, regarded as an additional geometric structure. The guiding principle of this approach is that \textbf{grading is ultimately related to homogeneity}. Assigning homogeneity degrees to coordinates in a consistent way is equivalent to fixing a global vector field, the \emph{weight vector field}. This approach is simple and substantially more general than most existing approaches to graded manifolds. In particular, the homogeneity degrees may be arbitrary real numbers, and the resulting category includes compact supermanifolds.

We systematically study homogeneity submanifolds, homogeneity Lie supergroups, tangent and cotangent lifts of homogeneity structures, homogeneous distributions and codistributions, as well as related notions such as double homogeneity. The main achievements of this framework include proofs of the homogeneous Poincar\'e Lemma, the homogeneous Frobenius Theorem, and the homogeneous symplectic Darboux Theorem, results that are of independent interest even in the purely even case.

\medskip\noindent
{\bf Keywords:} supermanifold; graded manifold; homogeneity; Poincar\'e Lemma; symplectic form; Darboux theorem.

\medskip\noindent
{\bf MSC 2020:} 58A50; 53D05; 53D10; 57N16 (Primary); 14A22; 17A30; 58A05 (Secondary).
\end{abstract}

\section{Introduction}
The theory of graded manifolds is a crucial mathematical tool in the BV and AKSZ formalisms of quantum field theory. Elements of supergeometry also appear in supersymmetric field theory, BRST theory, and supergravity.

\medskip
Many graded objects in the geometric literature are based on the assumption that a privileged family of (local) functions (sections of a sheaf, elements of a vector space, etc.) is assigned \emph{weights} (or \emph{degrees}), which are usually integers and, in supergeometry, are additionally equipped with a $\Z_2$-grading. To ensure that this family is sufficiently rich and that the corresponding grading is well defined, it should contain some local coordinates, and the weights should be preserved by transition maps of a prescribed class. This point is often left implicit in the literature, but it will be made precise in our approach.

\medskip\noindent
In the 1970s, \emph{supermanifolds} were introduced to provide a geometric framework for supersymmetry arising in physics. Berezin was the first to observe that, in analogy with classical algebra, analysis, and geometry, one can develop a mathematically rigorous theory of functions depending on anticommuting variables and, consequently, a theory of manifolds with both commuting and anticommuting coordinates. These ideas have found numerous applications in physics, particularly in quantum field theory. Following Berezin, the so-called Russian school of supergeometry developed the subject extensively. In the approaches of Berezin \cite{Berezin:1987}, Leites \cite{Leites:1980}, and Kostant \cite{Kostant:1977}, supermanifolds are defined as $\Z_2$-graded commutative locally ringed spaces (see also the books by Rogers \cite{Rogers:2007}, Tuynman \cite{Tuynman:1983}, and the supersymmetry-motivated book \cite{Carmeli:2011}).

Supermanifolds are often referred to simply as \emph{graded manifolds}. In this paper, however, we consistently use the term \emph{supermanifold}, reserving \emph{graded manifold} for manifolds equipped with gradings beyond the $\Z_2$-grading, such as $\Z$-graded manifolds. There is also a rapidly developing theory of $\Z_2^n$-graded supergeometry, see \cite{Bruce:2021,Bruce:2025,Bruce:2021a,Covolo:2016a,Covolo:2016b}, which generalizes the standard $\Z_2$-grading and finds applications in quantum field theory \cite{Aizawa:2024,Bruce:2020}.

\medskip\noindent
From the late 1990s onward, the introduction of a $\Z$-grading became essential in various contexts related to Poisson geometry, Lie algebroids, and Courant algebroids. Such structures appeared, for instance, in the work of Kontsevich \cite{Kontsevich:2003}, Roytenberg \cite{Roytenberg:2002}, and {\v{S}}evera \cite{Severa:2005}, who introduced the notion of an \emph{N-manifold}, i.e.\ a $\Z$-graded (in practice $\N$-graded) supermanifold whose parity is induced by the $\Z$-grading. Since then, $\Z$-graded manifolds have appeared in numerous works, for example \cite{Fairon:2017,Mehta:2006,Salnikov:2021,Kotov:2024,Vysoky:2022,Vysoky:2022a,Vysoky:2024}.

In many situations, the $\Z$-grading can be replaced by gradings indexed by elements of other commutative groups or even monoids \cite{Jiang:2023}, as was done for $\Z_2^n$-manifolds and, in our case, for the additive group $\R$. Our homogeneity supermanifolds are, in this sense, $\Z_2\times\R$-graded: the $\Z_2$-component encodes parity, while the $\R$-component encodes weight. These two gradings are compatible in the sense that they commute, forming a $\Z_2\times\R$-grading. Voronov originally considered general gradings not linked to parity \cite{Voronov:2019,Voronov:2002}, but later focused on $\Z$-gradings under additional assumptions, such as polynomial transition functions and the cylindrical nature of even coordinates of non-zero degree.

\medskip\noindent
Almost all recent papers on graded manifolds adopt the language of ringed spaces and sheaves of graded algebras, as developed in algebraic geometry. Differential geometry, however, originated as a framework for differential calculus on curved spaces, a feature that is largely absent from algebraic geometry. As a result, one often encounters ``graded manifolds'' whose graded algebras of local or global functions are too small to support a satisfactory differential calculus, the sheaf of polynomial functions on $\R$ being a trivial example. Moreover, the ringed-space approach is often inaccessible to physicists, for whom our framework is also intended.

\medskip\noindent
For this reason, we work with genuine supermanifolds equipped with the full apparatus of differential calculus, and treat grading as an additional geometric structure encoded by a vector field. Locally, any graded structure determines which functions are homogeneous. To ensure that the algebra generated by homogeneous functions is sufficiently rich, we require the existence of local homogeneous coordinates $(y^a)$. If $\deg(y^a)=w_a\in\R$, then the structure is encoded by the local weight vector field
$$
\n=\sum_aw_ay^a\pa_{y^a}.
$$
This allows us to define homogeneous functions (and tensor fields) of degree $w\in\R$ via the condition $\n(f)=wf$, even when $f$ is not polynomial in the coordinates. Requiring that transition maps preserve homogeneity ensures that the weight vector field is globally defined. Thus, a \emph{homogeneity structure} is represented by a global weight vector field, a viewpoint that is geometrically more transparent than that of abstract ringed spaces.

\medskip\noindent
\mn Our motivation is based on the well-known observation that grading is usually related to homogeneity, since in most approaches to graded supermanifolds one works with local homogeneous coordinates. Consequently, most graded supermanifolds appearing in the literature are homogeneity supermanifolds in our sense. This viewpoint is not new (see, e.g., the fundamental paper by Roytenberg \cite{Roytenberg:2002}), but it has to a large extent been replaced by the ringed-spaces approach.

Of course, in general, homogeneity is supplemented by additional structures, such as local graded vector superbundles combined with the requirement that transition maps be polynomial (or formal power series). Moreover, assuming that the weights of homogeneous functions take values in a discrete submonoid (usually $\N$ or $\Z$) of the additive group $\R$ provides further structure. In our case, homogeneous functions may have arbitrary real weights, so the corresponding algebra is not graded in the standard sense.

For supermanifolds, gradings should be compatible with parity only in the sense that every homogeneous coordinate can be chosen to be either even or odd. Note that this makes sense for arbitrary real weights (cf.\ \cite{Voronov:2002}) and is more general than determining parity from the weight, as is usually done for $\Z$-gradings. When arbitrary real weights are allowed, there is no natural way to assign parity based on them, as is possible in the $\Z$-graded case.

In the case of vector bundles $\zt:E\to M$, which are among the most thoroughly studied graded manifolds in differential geometry, the homogeneity structure is encoded by the Euler vector field $\n_E$ on $E$. The graded algebra, namely the algebra of functions that are polynomial along the fibers, is generated by linear functions, i.e.\ functions $f$ that are homogeneous of degree~1 and satisfy $\n_E(f)=f$ (Euler's Homogeneous Theorem). More generally, a tensor field $K$ is of degree $w\in\R$ if $\Ll_{\n_E}(K)=w\cdot K$. The base manifold $M$ is embedded in $E$ as the zero locus of the Euler vector field. In this case, the homogeneity structure defined by $\n$ is equivalent to an action of the monoid $(\R,\cdot)$ of multiplicative real numbers,
\be\label{ma} h:\R\ti E\to E,\quad h(t,y)=h_t(y),\ee
where $h_t$ is simply multiplication by $t\in\R$ in the vector bundle $E$. The map~(\ref{ma}) defines a monoid action in the sense that
\be\label{maa} h_1=\id_E,\quad\text{and}\quad h_t\circ h_s=h_{ts}.\ee
In particular, the bundle projection $\zt$ coincides with $h_0$. The vector field $\n_E$ is obtained as the generator of the restriction of $h$ to the $\R_+$-action $h:\R_+\ti E\to E$, where $\R_+$ denotes the subgroup of multiplicative $\R$ consisting of positive real numbers. To recover $h$ from $\n_E$, one observes that the $(\R_+,\cdot)$-action induced by $\n_E$ admits a unique smooth extension to a $(\R,\cdot)$-action. Indeed, by Euler's Homogeneous Theorem, functions that are differentiable and positively homogeneous of degree~1 are automatically linear, and hence homogeneous with respect to the entire group $(\R,\cdot)$.

This approach originated in \cite{Grabowski:2012}, where smooth actions of the monoid $(\R,\cdot)$ on manifolds (purely even in that case) were studied under the name \emph{graded bundles}. The corresponding objects turned out to carry canonical structures of fiber bundles with typical fiber $\R^k$, for some $k$, and the transition functions, by virtue of preserving homogeneity, turned out to be automatically polynomial in homogeneous coordinates of non-zero weight. Graded bundles were further investigated, with applications to geometric mechanics, in a series of papers \cite{Bruce:2015,Bruce:2015a,Bruce:2016,Bruce:2017a,Grabowska:2021}. A supersymmetric variant of this framework, which covers the geometry of N-manifolds, can be found—together with the $\C$-setting—in \cite{Jozwikowski:2016}. In fact, this was already the original setting of N-manifolds introduced by {\v S}evera \cite{Severa:2005}. This approach is useful for working with vector bundles and double vector bundles \cite{Grabowski:2009}. For instance, an N-manifold is, in this framework, a graded superbundle, that is, a supermanifold equipped with an action of $(\R,\cdot)$ such that $h_{-1}$ acts as the parity operator (cf.\ \cite{Jozwikowski:2016}).

This homogeneity-based approach to (super)vector bundles is effective and simplifies several concepts and proofs, including those concerning \emph{double vector bundles} \cite{Grabowski:2009}. It naturally leads to a \emph{compatibility condition} with other geometric structures, such as additional graded bundle structures, Lie groupoids, and Lie algebroids. This condition requires that each $h_t$ be a morphism in the relevant category (for example, a Lie groupoid morphism). In this way, one obtains constructive definitions of VB-groupoids, VB-algebroids, and related structures. For example, in the case of a homogeneity manifold with weight vector field $\n$, this compatibility condition means that a vector bundle $E$ in the category of homogeneity manifolds is a vector bundle whose Euler vector field $\n_E$ commutes with $\n$. By contrast, for \emph{double homogeneity} one requires a condition ensuring the existence of local coordinates that are bi-homogeneous, i.e.\ simultaneously homogeneous with respect to both homogeneity structures. This leads to a simple definition of a \emph{homogeneity vector superbundle}. It is currently unclear whether these two notions coincide (see Conjecture~\ref{conjecture}).

\mn The weight vector field is globally defined, so there is no intrinsic difficulty in preserving homogeneity under transition maps. However, we emphasize that ``preserving weights'' under transition maps does not mean that the set of weights of local coordinates is the same in every homogeneous chart. In fact, an important observation is that the systems of weights of local homogeneous coordinates may depend strongly on the chosen coordinate system. Since this observation is crucial for a proper understanding of our general concept, we illustrate it with a simple example.

\mn Consider the manifold $\R^\ti=\R\setminus\{0\}$ and choose $x$ as the standard coordinate inherited from $\R$. If we assign weight~$1$ to $x$, then the weight vector field is $\nabla=x\,\pa_x$. On the other hand, $y=1/x$ is a function of weight~$-1$, which can also be used as a coordinate on $\R^\ti$, and the corresponding weight vector field is $\n=-y\pa_y$. The transition function is
$$
\zf:\R^\ti\to\R^\ti\,,\quad y=\zf(x)=\frac{1}{x}\,.
$$
It is easy to see that $\zf_*(\nabla)=\nabla$. Hence, the weight vector field is preserved under the change of coordinates, so homogeneous functions have the same weight in both charts, even though the weights of the two coordinates differ.

On the other hand, we prove that in a neighbourhood of a point $m\in|M|$ in the body of $M$ at which the weight vector field vanishes, the weights of homogeneous coordinates around this point are the same, up to permutation. In particular, if the weights are integers for one coordinate system, then they are integers for any other coordinate system. This suggests that it may be interesting to consider a narrower class of homogeneity supermanifolds, assuming the existence of an atlas such that $\n$ vanishes at some point in each chart.

\medskip\noindent
In studying the geometry of general homogeneity supermanifolds, we also introduce and investigate double homogeneity supermanifolds, with homogeneity vector superbundles and homogeneity fiber superbundles as principal examples, as well as homogeneity Lie supergroups, tangent and cotangent lifts of homogeneity structures, and homogeneity distributions and codistributions. The main result in this part is the homogeneous Poincar\'e Lemma. The included examples illustrate how useful the concept of homogeneity is, for instance, in providing an elegant and effective treatment of vector superbundles.

\medskip\noindent
The next part is devoted to homogeneous distributions on homogeneity supermanifolds $(M,\n)$. These are defined as distributions $D\subset M$ that are homogeneous submanifolds of $\sT M$ with respect to the weight vector field given by the tangent lift $\dt\n$ of $\n$. In other words, $\dt\n$ is tangent to $D$. A fundamental result here is the homogeneous Frobenius Theorem.

\medskip\noindent
In the final part of the paper, we focus on homogeneous symplectic manifolds. Symplectic geometry provides a natural geometric framework for classical mechanics and for the integrability of differential equations. Numerous generalizations of symplectic geometry can be found in the literature. One non-classical aspect of supergeometry is the existence of odd geometric structures. In particular, the category of supermanifolds admits both even and odd contact \cite{Grabowski:2013} and symplectic structures; the latter are widely used in the BV--BRST formalism (see \cite{Khudaverdian:2004,Schwarz:1993,Schwarz:1996}). Another important concept is that of symplectic N-manifolds introduced by Roytenberg \cite{Roytenberg:2002}, where explicit descriptions of symplectic forms of degrees~1 and~2 are provided. This leads to an effective geometric approach to \emph{Courant algebroids} and their associated geometry.

\medskip\noindent
Analogs of the Darboux Theorem for general even and odd symplectic forms can be found in work by Shander \cite{Shander:1983} and Schwarz \cite{Schwarz:1990,Schwarz:1996}. Results of this type also appear in works by Kostant \cite{Kostant:1977} and Rogers \cite{Rogers:2007} on symplectic forms on supermanifolds, though without detailed proofs. The corresponding local Darboux forms are more complicated than their purely even counterparts, since wedge-commuting $1$-forms may exist. In this paper, motivated by the study of symplectic structures on $\Z_2^n$-manifolds by Bruce and Grabowski \cite{Bruce:2021}, we investigate homogeneous symplectic structures on homogeneity supermanifolds and analyze their local forms, proving the corresponding general graded Darboux Theorem.

Finally, we emphasize that our approach to homogeneity, and thus to graded supermanifolds, is rooted in traditional differential geometry, namely, in the use of collections of coordinate charts glued together by appropriate transition maps. We deliberately avoid direct references to sheaves and ringed spaces, which should make the paper more accessible to physicists. Readers who prefer standard (purely even) differential geometry may ignore the odd coordinates.

\section{Weight vector fields}
As we have already mentioned, for a given even vector field $\nabla$ on a supermanifold $M$ we can define (local) \emph{homogeneous functions $f$ of weight $w\in\R$} (and write $\w(f)=w$) as (local) smooth superfunctions on $M$ for which $\nabla(f)=w\cdot f$. We say that such a function is \emph{homogeneous of degree $\zl=(\zs,w)\in\Z_2\ti\R$} (and write $\deg(f)=\zl$) if, additionally, $f$ has the parity $\zs$. We will call $\zl$ \emph{even} if $\zs=0$ and \emph{odd} if $\zs=1$. Note, however, that local homogeneous functions form, generally, a `direct integral' rather than a direct sum of homogeneous subspaces. For instance, on $\R$ with $\n=\pa_x$ there exist homogeneous functions of any weight $w\in\R$, namely $e^{wx}$. Hence, the topology in the corresponding graded algebra is unclear. This also concerns the corresponding sheaves, making the ringed manifold approach to such graded structures problematic.

\begin{definition} \textbf{(Weight vector fields)}
An even vector field $\nabla$ on a supermanifold $M$ with the body (reduced manifold) $|M|$ we call a \emph{weight vector field} if, in a neighbourhood of every $p\in|M|$, we can always find local coordinates which are homogeneous with respect to $\nabla$. We do not assume that such coordinates vanish at $p$. Charts with homogeneous coordinates we call \emph{homogeneity charts}.
\end{definition}
\noindent The following is obvious.
\begin{proposition}
If $\nabla$ is a weight vector field on a supermanifold $M$, then there is an atlas on $M$ with local (super)coordinates $(x^i)$ such that
\be\label{ew} \nabla_M=\sum_iw_i\cdot x^i\,\pa_{x^i}\,,\ee
where $w_i\in\R$ represents the weight of the homogeneous coordinate $x^i$.
\end{proposition}
\begin{example}
According to the above Proposition, the vector field $\nabla_1=x\,\pa_x$ is a weight vector field on $\R$. On the other hand, the vector field $\nabla_2=x^2\,\pa_x$ is not a weight vector field. Indeed, if $y=y(x)$ is a homogeneous coordinate of weight $w\in\R$ in a neighbourhood of $0\in\R$, then $y'(0)=a\ne 0$ and $\nabla_2(y)=x^2\cdot y'(x)=w\,y(x)$. The derivative of the left side is 0 at $0\in\R$, so $w=0$, but $x^2\,y'(x)=0$ has only constant solutions $y(x)$. Note, however, that in a neighbourhood of $0\in\R$ there exist non-constant homogeneous functions with respect to $\nabla_2$. For instance, the function
$$f(x)=\begin{cases} 0 & \text{if} \quad x\le 0\\ e^{-1/x} & \text{if} \quad x>0\end{cases}$$
is homogeneous with respect to $\nabla_2$ with weight 1.
\end{example}
\no Before we go to a characterization of weight vector fields, let us note that if an even vector field $X$ on a supermanifold $M$ of dimension $(n|m)$ vanishes at $p\in|M|$, then there is a well-defined differential $D_pX\in\End(\sT_pM)\simeq\on{gl}(n,m)$. Indeed, we can view $D_pX$ as $\sT_pX:\sT_p M\to\sV_{0_p}\sT_pM$, identifying canonically the vertical part $\sV_{0_p}\sT_pM$ of the tangent space $\sT_{0_p}\sT_pM$ with $\sT_pM$. Note that this works only because $X(p)=0$. In this case, the condition (\ref{ew}) means that $D_p\n$ is diagonal.

\mn If, in turn, $X(p)\ne 0$, then it is well known that $X$ can be locally written as $\pa_{x^1}$ in a coordinate system $(x^a)$. This is a version of the straightening out theorem for supermanifolds and even vector fields \cite[Theorem 1]{Shander:1980} (see also \cite{Garnier:2013}). For odd vector fields, it works under the integrability condition $[X,X]=0$. If $\pa_{x^1}$ is a weight vector field, then $x^1$ must be even, and we can change this coordinate system to $(\tilde x^1,x^2,\dots,x^n)$, where $\tilde x^1=e^{x^1}$. In this system, $\tilde x^1$ is invertible, and the weights of coordinates are 1 for $\tilde x^1$ and 0 for $x^i$, $i>1$.
This means that we can always choose coordinates with weights $(1,0,\dots,0)$ in a neighbourhood of a point of $|M|$, where the weight vector field is non-vanishing. We can as well use the coordinate $e^{w\cdot x^1}$, with $w\ne 0$, which is of degree $w$.
\begin{proposition}\label{Sh}
An even vector field $\n$ on a supermanifold $M$ is a weight vector field if and only if $\n(M)$ is locally linear with $D_m\n$ diagonal for each $m\in|M|$ in the zero-locus of $\n$. In particular, non-vanishing even vector fields are weight vector fields.

\mn If in local homogeneous coordinates $(x^i)$ we have
$$\nabla_M=\sum_{i=1}^nw_i\cdot x^i\,\pa_{x^i},$$
then:
\begin{itemize}
\item In the case $\nabla_M(m)=0$, all weights of systems of homogeneous coordinates around $m$ are
the same for each system of homogeneous coordinates up to permutations among the weights of even and odd coordinates separately. In particular, if $\,\nabla_M$ admits local homogeneous coordinates with weights in a subset $\zG\subset\R$ (e.g., with integer weights), then all homogeneous coordinate systems around $m$ have weights in $\zG$ (have integer weights).

\item In the case $\nabla(m)\ne 0$, the vector $(w'_i)\in\R^{p+q}$ consists of weights of a system of local homogeneous coordinates in a neighbourhood of $m$ if and only if not all weights of even coordinates are 0. In particular, for any $w,v_i\in\R$, $i=2,\dots,n$, $w\ne 0$, there are coordinates $(x,z^i)$, $i=2,\dots,\dots n$ in a neighbourhood of $m$, such that $x$ is even, $x(m)=1$, $z^i(m)=0$, and
    $$\nabla_M=w\cdot x\,\pa_x+\sum_{i=2}^nv_i\cdot z^i\,\pa_{z^i}\,.$$
\end{itemize}
\end{proposition}
\begin{proof}
In the first case, let us take local homogeneous coordinates $(y^a,\zx^\za)$, $y^a$ even and $\zx^\za$ odd, with weights $u_a$ and $v_i$, so that
$$\nabla=\sum_au_a\cdot y^a\pa_{y^a}+\sum_\za v_\za\cdot\zx^\za\pa_{\zx^\za}$$
and $m=0$ in these coordinates.
For another homogeneous coordinate system $(\bar y^a,\bar\zx^\za)$, we have the Taylor expansions
\beas \bar y^a&=&t^a_b\cdot y^b+o(y,\zx)\\
\bar \zx^\za&=&s^\za_\zb\cdot\zx^\zb+ o(y,\zx).\,,
\eeas
where $[t^a_b]$ and $[s^\za_\zb]$ are invertible real matrices. It is clear that $\bar y^a$ is homogeneous of weight $\bar u_a$ only if $\bar u_a=u_b$ for all $b$ for which $t^a_b\ne 0$. Thus, any weight $\bar u_a$ belongs to the set of weights of $(y^b)$. On the other hand, as the matrix $[t^a_b]$ is invertible, any weight $u_b$ is present in the sequence $(\bar u_a)$ as many times as it appears in the sequence $(u_a)$. All this works also for odd coordinates.

\medskip\noindent
In the second case, there is an even coordinate, say $x^1$, such that $w_1\cdot x^1(m)\ne 0$. Hence, $x^1/x^1(m)=1$ and $x^1/x^1(m)$ is positive in a neighbourhood of $m$. Moreover, the function $t=(x^1/x^1(m))^{1/w_1}$ is an even function of weight 1 with $t(m)=1$ (so $t$ is positive in a neighbourhood of $m$) and $(\xd t)(m)\ne 0$ is proportional to $\xd x^1$, so that we can use $t$ as a coordinate complementary to $(x^i)$, $i>1$. If we take $w\ne 0$ and $v_j\in\R$, $j=2,\dots,n$, then putting $x=t^w$,
$z^i=x^i\cdot t^{(v_i-w_i)}-x^i(m)$ for $i=2,\dots,n$, we get
a system $(x,z^i)$ of homogeneous local coordinates for which  $x(m)=1$, $x$ is of degree $(0,w)$, and the coordinate $z^i$ has weight $v^i$ and vanishes at $m$.

\end{proof}
\begin{corollary}\label{cor}
If $\nabla_M(m)\ne 0$, then there exist local coordinates $(t,z^2,\dots,z^{p+q})$ in a neighbourhood of $m$ such that $t$ is even and $t(m)=1$, $z^i(m)=0$, and in these coordinates $\nabla_M=t\,\pa_t$.
A function $g(t,z^a)$ is homogeneous of degree $\zl=(\zs,w)$, if and only if
$$g(t,z)=t^w\cdot h(z)\,,$$
where $h$ is any smooth function in variables $z^i$  of parity $\zs$.
\end{corollary}
\begin{proof}
The existence of a coordinate system of the desired form was established in Proposition \ref{Sh}.
The function $g$ has weight $w$ if and only if
$$\nabla_M(g)(t,z)=t\,\frac{\pa g}{\pa t}\,(t,z)=w\cdot g(t,z)\,.$$
This is, actually, an ordinary differential equation with respect to $t$ with parameters $(z^i)$.
All solutions are of the form $t^w\cdot h(z)$, and the corollary follows.

\end{proof}
\no It is interesting that for some weight vector fields there exist non-zero and simultaneously flat at some point homogeneous functions of arbitrary degree $w\in\R$, as shown in the following example.
\begin{example}
Take $\R^2$ with coordinates $(x,y)$, where $x$ is of degree $1$ and $y$ is of degree $-1$, so that the weight vector field is $\nabla=x\pa_x-y\pa_y$. Take a non-zero function $\zf:\R\to\R$ which is flat ($\zf$ and all derivatives of $\zf$ vanish) at the point $0\in\R$,  but $\zf(t)$ is positive for $t\ne 0$. Then, $f(x,y)=\zf(xy)$ is of degree $0$ with respect to $\nabla$, but it is not constant. The function $f_1(x,y)=x\cdot\zf(xy)$ is of degree $1$, but it is not a polynomial in coordinates, while the function
$$\begin{cases}f_d(x,y)=|x|^w\zf(xy)\quad\text{for}\quad x\ne 0;\\
f_d(x,y)=0 \qquad\text{for}\quad x=0\,,
\end{cases}$$
where $w\in \R$, is clearly smooth ($\zf$ is flat at $0$) and of degree $w\in\R$.
\end{example}

\begin{remark} Proposition \ref{Sh} shows that the restrictions on weight vector fields are not very tight. Especially, in charts where $\nabla_M$ is non-vanishing, the situation is rather pathological and not very interesting. In particular, any nowhere-vanishing vector field on a supermanifold $M$ is a weight vector field with homogeneous coordinates having any weights we declare (but at least one even coordinate must have a non-zero weight), so for many purposes, we may restrict our concept of a weight vector field. For instance, one can assume that the atlas consists of homogeneous charts with the additional assumption that $\n$ vanishes at some point in every chart. Then the degrees of homogeneity of coordinates will be locally fixed (up to a permutation, of course). We will postpone a deeper study of this case for a separate paper.
\end{remark}
\no Now, we will present a description of all possible weights of local homogeneous functions on a homogeneity supermanifold $M$ in a neighbourhood of a certain $m\in|M|$.
As the case $\nabla_M(m)\ne 0$ is already completely described (all real weights are possible), we will concentrate on the case $\nabla_M(m)=0$, which is more complicated. Since local smooth functions on $M$ are polynomials in odd variables $\zx^i$ and coefficients are local smooth functions on $|M|$, and since monomials in $\zx^i$ are always homogeneous, the knowledge about homogeneous functions on $M$ boils down to knowing the functions on $(|M|,\nabla_{|M|})$. Therefore, we can assume that $M$ is purely even.

\mn Let
$$\nabla_M=\sum_{a=1}^pw_a\cdot x^a\,\pa_{x^a}$$
be the local form of a weight vector field on a purely even homogeneity manifold $M$ of dimension $p$ in a neighbourhood of $m\in M$ such that $\nabla_M(m)=0$.
With $W_{\nabla_M}(m)$ we denote the set
$$W_{\nabla_M}(m)=\bigg\{ \sum_a n^a\cdot w_a\,|\, n^a\in\N \bigg\} \,.$$
According to Proposition \ref{Sh}, this definition is correct, i.e., does not depend on the choice of local homogeneous coordinates in a neighbourhood of $m$. It is obvious that all monomials in variables $x^a$ are homogeneous with weights in $W_{\nabla_M}(m)$ and that $W_{\nabla_M}(m)\subset\Z$ if all $w_a$ are integers.

We can divide the coordinates $(x^a)$ into two families: $(y^\za)$ with non-zero weights, and $(z^j)$ with weights $0$, so that
\be\label{wvfa}\nabla_M=\sum_\za w_\za\cdot y^\za\,\pa_{y^\za},\ee
and $y^\za(m)=0$ for all $\za$. We can also assume that $z^j(m)=0$ for all $j$. In particular, any homogeneous function with a non-zero weight is vanishing on the submanifold $y^a=0$.
\begin{proposition}\label{pos} Consider local supercoordinates $(y^\za,z^j)$ and the weight vector field (\ref{wvfa}). Then, we have the following possibilities.
\begin{enumerate}
\item If all $w_\za>0$ (resp., $w_\za<0$), then any homogeneous function $f$ near $m$ is a polynomial in variables $y^\za$ with coefficients being functions in variables $z^j$. The weights of homogeneous functions around $m$ are precisely the elements of $W_{\nabla_M}(m)$ which is in this case a discrete subset of $\R_{\ge 0}$ (resp., $\R_{\le 0}$). In particular, allowed are only non-negative (resp., non-positive) weights of homogeneous functions.
\item If there are positive and negative $w_\za$, say, $w_1>0$ and $w_2<0$, then there are local homogeneous functions of weight 0 which are not of the form $g(z)$. If a local homogeneous function $f$ of weight $w$ is not flat at $m$ modulo a constant (i.e., a non-zero derivative of $f$ does not vanish at $m$), then $w\in W_{\nabla_M}(m)$. On the other hand, there are local homogeneous functions of arbitrary weight $w\in\R$ and flat at $m$.
\end{enumerate}
\end{proposition}
\begin{proof}
\textbf{1.} First, we will consider the case of positive weights. Let $h_t(y^\za,z^j)=(t^{w_\za}\cdot y^\za, z^j)$ for $t\in\R$ close to 1. Since all $w_\za$ are positive, we can choose a neighbourhood $U$ of $m$ such that $h_t$ in $U$ is defined for all $0<t\le 1$. It is easy to see that $f(y,z)$ is homogeneous of degree $w\in\R$ if and only if $f\circ h_t=t^w\circ f$.
We will prove now that all homogeneous functions of weight $w\le 0$ are of the form $f(z)$, so that $w=0$.

Let $(y^\za)=(u^i,\zx^a)$ (resp., $(z^j)=(v^I,\zh^A)$), where the coordinates $u^i$ are even and $\zx^a$ are odd
(resp. $v^I$ are even and $\zh^A$ are odd). Any function $f(y,z)$ is a polynomial in coordinates $(\zx,\zh)$ with coefficients of the form $g(u,v)$. Consider a monomial $$P(y,z)=g(u^i,v^I)\,\zx^{a_1}\cdots\zx^{a_k}\,\zh^{A_1}\cdots\zh^{A_l}$$
being a summand in the polynomial $f$. If $f$ is homogeneous with weight $w$, then $P\circ h_t=t^w\cdot P$, so
$$(P\circ h_t)(y,z)=(g\circ h_t)(u,v)\,\zx^{a_1}\cdots\zx^{a_k}\,\zh^{A_1}\cdots\zh^{A_l}\cdot t^{(w_{a_1}+\dots +w_{a_k})}=t^w\cdot P(y,z)\,,$$
so that  $(g\circ h_t)=t^{w'}\cdot g$, where $w'=w-w_{a_1}-\dots -w_{a_k}$. If $w<0$, then clearly $w'<0$ (as $w_{a_s}> 0$). In particular,
$$(g\circ h_t)(u,v)=g(t^{w_i}\cdot u^i,v)=t^{w'}\cdot g(u,v)\,.$$
Note that $g$ is a (local) smooth function on $M$. If $g$ is nonzero, there is $(u_0,v_0)$ such that $g(u_0,v_0)\ne 0$.
But $w_i>0$, so
$$g(0,v_0)=\lim_{t\to 0+}(g\circ h_t)(u_0,v_0)\ne\lim_{t\to 0+}t^{w'}\,g(u_0,v_0)=\pm\infty\,;$$
a contradiction. Hence, $w'=0$, so consequently $w=0$ and the monomial $P$ does not depend on $(u^i,\zx^a)$.

\mn Let now $a>0$ be the minimal number in the finite set of positive reals $\{ w_\za\}$. We will show inductively with respect to $n$ that all homogeneous functions of weight $\le na$ are polynomial. In the case $n=0$, we have just proved it. So, suppose
that all homogeneous functions of weight $\le na$ are polynomials in $y^\za$ with coefficients in the form $g(z)$ and let $f(y^\za,z^j)$ be homogeneous with weight $w\le (n+1)a$. Then, according to Lemma \ref{ul}, the partial derivative $\frac{\pa f}{\pa y^k}(y,z)$ is of weight $w-w_k\le (n+1)a-w_k\le na$, so it is a polynomial in variables $y^a$ with coefficients of the form $g(z)$. We will show that $f$ itself is of the desired form, using another (this time finite) induction with respect to the number $r$ of variables $y^\za$.

The case $r=0$ is trivial, so assume the inductive assumption for $r$. Having now $(r+1)$ variables $y^1,\dots,y^{r+1}$, consider a function $f(y^\za,z^j)$ whose partial derivative with respect to $y^1$ is a homogeneous of weight $(w-w_1)$, polynomial in variables $y^\za$ and coefficients being homogeneous functions in $(z^j)$. In particular,
$$\frac{\pa f}{\pa y^1}(y,z^j)=\sum_ig_i(z)y^{a^i}\,,$$
where $a^i$ are multi-indices, $a^i=(a^i_1,\dots,a^i_{r+1})$, and
$$y^{a^i}=(y^1)^{a^i_1}\cdots (y^{r+1})^{a^i_{r+1}}\,.$$
It is easy to see that $y^{a^i}$ is homogeneous with weight $|a^i|=\sum_{\za=1}^{r+1}a^i_\za\cdot w_\za$, so that
$|a^i|=w-w_1$ for all $i$. It is clear now that
$$f(y^\za,z^j)=\sum_i\frac{g_i(z)}{a^i_1+1}y^{a^i}\cdot y^1 + f_1(y^2,\dots,y^{r+1},z^j),$$
for some function $f_1$ which is of weight $w$ and depends only on $r$ variables from $(y^\za)$. Now, we apply the inductive assumption. All this implies  immediately that the set of possible weights of local homogeneous functions is exactly $W_{\nabla_M}(m)$. The proof of the case $w_\za<0$ is completely analogous.

\medskip\noindent
\textbf{2.} Let $a=-w_1/w_2$. Then, $a>0$ and the function $g(y,z)=y^1\,|y^2|^a$ is homogeneous of weight 0 on the open subset defined by $y^2\ne 0$. Of course, this function is generally not smooth at points for which $y^2=0$, but we can compose it with a smooth function $h:\R\to\R$ which is flat at 0. Then, the function $f(y,z)=h(y^1\,|y^2|^a)$ is smooth and flat at $m$. Actually, it is flat at all points $(y,z)$ with $y^1\cdot y^2=0$. Lemma \ref{ul} implies now that $f$ is homogeneous of weight 0. Let us take any $w\in\R$. The function
$$f_w(y,z)=(y^1)^{w/w_1}f(y,z)$$ is clearly smooth and of weight $w$.

On the other hand, if $f(y,z)$, vanishing at $m$, is not flat at $m$, then there is a Taylor decomposition $f(y,z)=P(y,z)+o(y,z)$ of $f$ around $0$, with $P$ being a non-zero polynomial of degree $r$ in variables $(y,z)$, and $o(y,z)$ being infinitesimal of order $< r$. We have $$\nabla_M(f)=\nabla_M(P)(y,z)+\nabla_M(o)(y,z)=w\cdot(P(y,z)+o(y,z))\,.$$
Since $\nabla_M(m)=0$, also $\nabla_M(o)(y,z)=o'(y,z)$, is infinitesimal of order $<r$. But the Taylor polynomial of degree $r$ is  uniquely determined up to order $r$, so $\nabla_M(P)=w\cdot P$, i.e., the polynomial $P$ is homogeneous of degree $w$. All monomials in variables $(y,z)$ are homogeneous functions with weights in $W_{\nabla_M}(m)$, so $w\in W_{\nabla_M}(m)$.

\end{proof}
\section{Homogeneity supermanifolds}
\begin{definition} \textbf{(Homogeneity supermanifolds)}
A supermanifold $M$ equipped with a weight vector field $\nabla_M$, we will call a \emph{homogeneity supermanifold}. If $\zG\subset\R$ is a subset of $\R$ for which there exists a distinguished atlas of homogeneity charts with coordinates having weights in $\zG$, then we will say that $(M,\nabla_M)$ is a \emph{$\zG$-homogeneity supermanifold}. The corresponding charts we call \emph{allowed charts} and the corresponding coordinates \emph{allowed coordinates}.
\end{definition}

\begin{remark} Note that, since $\nabla_M$ is a well-defined global geometric object on $M$, the transition functions respect $\nabla_M$, so they are in this sense `weight-preserving'. In particular, as the property of being a (local) homogeneous function of weight $w$ with respect to $\nabla_M$ does not depend on the choice of coordinates, homogeneous functions of weight $w$ are the same for local coordinate systems on their intersection. Consequently, one can construct homogeneity supermanifolds by choosing an atlas, declaring the weights of local coordinates, and demanding that the corresponding transition maps preserve the local weight vector fields, which therefore give rise to a globally defined weight vector field.
\end{remark}
\begin{remark}
This idea of  homogeneity supermanifolds is essentially known (see, e.g., \cite{Roytenberg:2002,Severa:2005}, where the corresponding objects are called \emph{graded manifolds}), although usually restricted to (mainly non-negative) integers. In the mentioned papers, only those $\N$-homogeneity supermanifolds are studied for which the parity of homogeneous coordinates is determined by the weight; they are called \emph{N-manifolds}. They have also been studied in \cite{Grabowski:2012,Jozwikowski:2016} in a simplifying framework. Determining the parity by the weight is in our picture, and an additional assumption, we do not consider in this paper.

A more general concept of a graded manifold, probably the closest to ours, was that by Voronov \cite{Voronov:2002} (cf. also \cite{Mehta:2006}). In his approach, acceptable weights are also integers but not linked to parity. The allowed charts have the form $U\ti\R^{k|l}$, where the fiber coordinates have non-zero integer degrees and $U$ is an open subset of $\R^{n|m}$ with the trivial homogeneity structure, so the manifold has a structure of a fiber bundle (Voronov says that even coordinates of non-zero weight are `cylindrical'). There is an additional strong requirement that the allowed transition maps are polynomial in fiber coordinates, so we deal with a polynomial fiber bundle, so not all homogeneous charts are allowed. Of course, our approach covers all these examples based on homogeneous coordinates.

Note that graded manifolds, defined simply as graded ringed spaces, are concepts that are too weak to have a proper differential calculus. For instance, working on $\R$ with the sheaf of polynomial functions, we are unable to define exponential or trigonometric functions, as they are not sections of the sheaf. We must consider our sheaf as a subsheaf of the structure sheaf of the supermanifold. It is therefore an additional structure on the supermanifold, which, in most cases, can be viewed as a homogeneity structure in our sense.
\end{remark}
\begin{definition} \textbf{(Morphisms)}
A \emph{morphism} of a  homogeneity supermanifold $(M,\nabla_M)$ into a  homogeneity supermanifold $(N,\nabla_N)$ is a morphism $\zf:M\to N$ of supermanifolds which relates $\nabla_M$ with $\nabla_N$. In other words,
the pullbacks of (local) homogeneous functions on $N$ are homogeneous on $M$ with the same degree. More generally, a morphism $\zf:M\to N$ of supermanifolds  is called \emph{of weight $\zl\in\R$} if the pullbacks of (local) homogeneous functions of weight $w\in\R$ are homogeneous of weight $w+\zl$.
\end{definition}
\noindent The following is obvious.
\begin{theorem}
Homogeneity supermanifolds with morphisms defined as above form a category which we will denote $\catname{HSMan}$.
\end{theorem}
\begin{proposition} \textbf{(Cartesian products)}
If $(M,\nabla_M)$ and $(N,\nabla_N)$ are  homogeneity supermanifolds, then the Cartesian product $M\ti N$ carries a canonical structure of a  homogeneity supermanifold whose weight vector field $\nabla_{M\ti N}$ is defined as the sum $\tilde\nabla_M+\tilde\nabla_N$, where $\tilde\nabla_M$  (resp. $\tilde\nabla_N$) is understood as the unique vector field on $M\ti N$ whose projections on $M$ and $N$ are $\nabla_M$ and $0$ (resp. $0$ and $\nabla_N)$). We will just write, with some abuse of notation, $\nabla_{M\ti N}=\nabla_M+\nabla_N$, and we will call $(M\ti N,\nabla_M+\nabla_N)$ the \emph{Cartesian product} of  homogeneity supermanifolds $(M,\nabla_M)$ and $(N,\nabla_N)$.
\end{proposition}
\begin{proof}
Indeed, if $(x^a)$ and $(y^j)$ are homogeneous coordinates on $M$ and $N$ with weights $(w_a)$ and $(v_j)$, respectively, then in local coordinates
$(x^a,y^j)$ on $M\ti N$ the weight vector field $\nabla_{M\ti N}$ reads
$$\nabla_{M\ti N}=\sum_aw_a\cdot x^a\pa_{x^a}+\sum_jv_j\cdot y^j\pa_{y^j}\,.$$
\end{proof}
\begin{definition}
A submanifold $N$ of a  homogeneity supermanifold $M$ we call \emph{homogeneous submanifold} if $\nabla_M$ is tangent to $N$.
\end{definition}
\begin{proposition} If $N$ is a homogeneous submanifold  of a  homogeneity supermanifold $M$, then $N$ is canonically a  homogeneity supermanifold itself with the weight vector field $\nabla_N$ being the restriction of $\nabla_M$ to $N$, i.e., $\nabla_N=\nabla_M\big|_N$.
\end{proposition}
\begin{proof} Suppose that the total dimensions are: $\dim(N)=n$ and $\dim(M)=m$. For $m=n$ the proposition is obvious, so let us consider $m>n$.
There is a covering of $M$ with some coordinate neighbourhoods $U$ with coordinates $(y^1,\dots,y^m)$ such that
$N\cap U$ is defined by the system of equations $y^{n+1}=0,\dots,y^m=0$. Of course, these coordinates are not \emph{a priori} homogeneous. Making $U$ smaller if necessary, we can have on $U$ also homogeneous coordinates $(x^1,\dots,x^m)$, so that $\nabla_M=\sum_aw_a\cdot x^a\,\pa_{x^a}$. There are $n$ coordinates from $(x^a)$, say, $x^1,\dots,x^n$ such that their restrictions $\tilde x^1,\dots,\tilde x^n$ to $N$ give a system of coordinates on $U\cap N$. As functions in coordinates $y^a$, they read $\tilde x^i=x^i(y^1,\dots,y^n,0,\dots,0)$. But $x^i$ is homogeneous with weight $w_i$, so that
$$\nabla_M\big|_N(\tilde x^i)=\nabla_M\left(x^i(y^1,\dots,y^n,0,\dots,0)\right)=w_i\cdot x^i(y^1,\dots,y^n,0,\dots,0)=
w_i\cdot \tilde x^i\,.$$
Hence, $(\tilde x^1,\dots,\tilde x^n)$ are local coordinates in $N$ which are homogeneous with respect to $\nabla_M\big|_N$.

\end{proof}
\no It is easy to see the following.
\begin{proposition}
For each  homogeneity supermanifold $M$, the weight vector field $\nabla_M$ induces on the \emph{body} (called also the \emph{reduced manifold}) $|M|$ of $M$ a weight vector field $\nabla_{|M|}$, which for homogeneous local coordinates $(y^i,\zx^a)$($y^i$ even and $\zx^a$ odd) on $M$ takes the form
$$\nabla_{|M|}=\sum_{i}w_i\,y^i\pa_{y^i}$$
and makes $|M|$ into a purely even homogeneity manifold. In other words, $(|M|,\nabla_{|M|})$ is a homogeneous submanifold of $(M,\nabla_M)$.
\end{proposition}
\noindent Let now $M_0\subset|M|$ be the set of zeros of the vector field $\nabla_{|M|}$.
\begin{proposition}
The set $M_0$ is, actually, a submanifold of the even manifold $|M|$, thus also an even submanifold of $M$.  The manifold $M_0$ is a homogeneous submanifold in $M$ with a trivial homogeneity structure.
\end{proposition}
\begin{proof}
The manifold $|M|$ is a purely even homogeneity manifold with the weight vector field written locally as $\nabla_{|M|}=\sum_{i}w_i\,y^i\,\pa_{y^i}$, where $(y^i)$ are even homogeneous local coordinates in $M$, so homogeneous coordinates in $|M|$. Hence, the submanifold $M_0$ of $|M|$ is described by the system of equations $y^i=0$ for that $ i$ for which $w_i\ne 0$. As a submanifold in $M$, the manifold $|M|$ is described in homogeneous coordinates $(y^i,\zx^a)$ by the system of equations $\zx^a=0$  and $y^i=0$ for those $i$ for which $w_i\ne 0$.

\end{proof}
\no Since nowhere-vanishing vector fields are automatically weight vector fields, $M_0$ may be empty.
A useful Lemma we will use many times is the following.
\begin{lemma}\label{ul}
Let $M$ be a homogeneity supermanifold, and let $(x^a)$ be local homogeneous coordinates with degrees $(\zl_a)$ in a neighbourhood of $m\in|M|$.
\begin{enumerate}
\item If $g$ is an even homogeneous function of weight 0, and $h:\R\to\R$ is smooth, then $f=h\circ g$ is also an even homogeneous function of weight 0.
\item If $f$ and $g$ are homogeneous functions of degrees $\zl_f$ and $\zl_g$, then $f\cdot g$ is homogeneous of degree $\zl_f+\zl_g$.
\item If $f$ is a homogeneous function of degree $\zl$, then $\pa_{x^k}(f)$ is of degree $\zl-\zl_k$.
\item If $f$ is a homogeneous even and positive function of weight $w$, then $f^v$ is homogeneous even and positive of weight $vw$ for any $v\in\R$.
\end{enumerate}
\end{lemma}
\begin{proof}
For the first statement, we calculate
$$\nabla_M(f)(x)=\sum_a w_a\cdot x^a\,\frac{\pa f}{\pa x^a}(x)=\sum_a h'(g(x))\cdot w_a\cdot x^a\cdot\frac{\pa g}{\pa x^a}(x)=h'(g(x))\cdot\nabla_M(g)(x)=0\,.$$
The second statement follows immediately from the Leibniz rule.
As for the third one, we have
$$w\cdot\pa_{x^k}(f)=\pa_{x^k}(\nabla_M(f))=\nabla_M(\pa_{x^k}(f)+[\pa_{x^k},\nabla_M](f)=
\nabla_M(\pa_{x^k}(f))+w_k\cdot\pa_{x^k}(f)\,,$$
so
$$\nabla_M(\pa_{x^k}(f))=(w-w_k)\cdot\pa_{x^k}(f)\,.$$
And finally,
$$\nabla_M(f^v)=v\cdot f^{v-1}\cdot \nabla_M(f)=v\cdot f^{v-1}\cdot w\cdot f=(v\,w)\cdot f^v\,.$$

\end{proof}

\mn The notion of the weight of homogeneity can be easily extended from just functions to arbitrary tensor fields by using the Lie derivative $\Ll_{\nabla_M}$.
\begin{definition} Let $\n$ be a weight vector field on a supermanifold $M$. We say that a (local) tensor field $K$ is of weight $w\in\R$ if
$$\Ll_{\nabla_M}(K)=w\cdot K.$$
If additionally $K$ has a defined parity, say $\zs$, then we call it \emph{homogeneous of degree $\zl=(\zs,w)\in\Z_2\ti\R$}.
\end{definition}
\begin{remark}\label{rw} \textbf{(Flows)}
It is not widely known that the tangent bundle $\sT M$ of a supermanifold $M$, \emph{via} the functor of points, can be identified with the first jet bundle of smooth curves in $M$, i.e., smooth morphisms $\zg:\R\to M$ (see \cite{Bruce:2014}). Moreover, every even vector field $Y$ induces a flow $\Exp(t\,Y)$ of local diffeomorphisms of $M$ \cite[Chapter V]{Tuynman:1983} such that
$$\left.\frac{\rmd}{\rmd s}\right|_{s=0}\Exp(s Y)(x)=Y(x)\,.$$
For a tensor field $K$, the Lie derivative $\Ll_Y(K)$ can be understood as the infinitesimal action of the one-parameter group $\Exp(sY)$ of local diffeomorphisms on the tensor $K$:
$$\left.\frac{\rmd}{\rmd s}\right|_{s=0}\Exp(s Y)^\star (K)=\Ll_Y K\,,$$
where $\Exp(s Y)^\star$ denotes the action (pullback) of the local diffeomorphism $\Exp(s Y)$ on tensor fields.
More precisely, $\Exp(s Y)^\star (K)$ is the standard pullback if $K$ is covariant, and it is understood as
$\Exp(-sY)_* (K)$ if $K$ is contravariant, extended then in an obvious way to all tensor fields.
Therefore, the Lie derivative $\Ll_Y K$ is well defined for any tensor field $K$, not only for even or odd $K$. Moreover, the bracket $[Y,X]$ of vector fields is well defined, although vector fields on supermanifolds do not generally represent (graded) derivations in the algebra $\cA(M)$ of superfunctions on $M$, which is sometimes stated in textbooks in a misleading way. In the case of the weight vector field $Y=\nabla_M$ (see (\ref{ew})), the corresponding flow reads
\be
g_s(x)=\Exp(s \nabla_M)(x)=\left(e^{w_i\cdot s}\cdot x^i\right)\,,
\ee
for $s$ sufficiently close to 0, with `sufficiently' depending on the neighbourhood in the body part of homogeneous coordinates.
It is easy to see that a function $f$ is homogeneous of weight $w$ if and only if $f\circ g_s=e^{w\cdot s}f$.
It will be convenient to change the parameter $s\in\R$ into the parameter $t=e^s>0$ and put
$$h_t(z)=(t^{w_i}\cdot z^i)\,.$$
Then, $h_t\circ h_u=h_{tu}$, and the condition for homogeneity of $f(z)$ with weight $w$ is
\be f\circ h_t=t^w\cdot f\,.\ee
The vector field is \emph{complete} if $h_t$ is globally defined for all $t>0$.
\end{remark}
\begin{corollary}
A weight vector field $\n$ is complete if and only if $\n$ is the generator of a smooth action $h:\R_+\ti M\to M$ of the multiplicative group $\R_+$ of positive reals on $M$, in the sense that
$$\left.\frac{\rmd}{\rmd t}\right|_{t=1}h_t=\n.$$
\end{corollary}
\no Our method of defining graded manifold \emph{via} homogeneity structures is very effective and leads immediately to a concept of multi-homogeneity structures.
\begin{definition}
A \emph{multi-homogeneity structure} on a supermanifold $M$ is a finite sequence $(\n_1,\dots,\n_k)$ of weight vector fields on $M$ which are \emph{compatible}, i.e., they commute pairwise, $[\n_i,\n_j]=0$. In this case, we call it a \emph{$k$-tuple homogeneity structure}. Tensor fields which are homogeneous with respect to each $\n_i$ with weight $w_i$ we call \emph{homogeneous} with $k$-weight $(w_1,\dots,w_k)\in\R^k$. Morphisms of $k$-tuple homogeneity manifolds are morphisms of supermanifolds relating $\n_i^1$ with $\n_i^2$ for all $i$.
\end{definition}
\no Any $k$-tuple homogeneity structure we will view as an $\R^k$-grading. Of course, in practice, some parts are $\Z$ or $\N$-gradings. A particularly important feature for the geometry of vector bundles is the $\N\ti\R$-grading.
\begin{example} Consider on $M=\R^p\ti\R^q$ with coordinates $(x^a,y^i)$ compatible weight vector fields
$$\n_1=\sum_iy^i\pa_{y^i},\quad \n_2=\sum_aw_ax^a\pa_{x^a}+\sum_iv_iy^i\pa_{y^i},$$
where $w_a,v_i\in \Z$. This gives an $\N\ti\Z$-gradation on $M$. The $\N$-part of the gradation is responsible for the obvious trivial vector bundle structure on $M$,
$$\zt:\R^p\ti\R^q\to\R^q.$$
\end{example}

\section{Examples}
\begin{example} \textbf{(Trivial example)} Any supermanifold $M$ is a homogeneity supermanifold with the trivial (zero) weight vector field.
\end{example}
\begin{example} \textbf{(Vector superbundles)} Consider the trivial vector superbundle of rank $(n|m)$ over a supermanifold $M$, $E=M\ti\R^{n|m}$, with the obvious surjective submersion $\zt:E\to M$. If $(y^a,\zh^j)$ are the standard linear coordinates in $\R^{n|m}$, then
$$\n_\cE=\sum_ay^a\pa_{y^a}+\sum_j\zh^j\pa_{\zh^j}$$
is a weight vector field on $E$, called the \emph{Euler vector field}, and the coordinates $(y^a,\zh^j)$ are all of weight 1. Hence, $E$ is canonically a homogeneity manifold, and $M$ is a submanifold in $E$ defined as the zero locus of all 1-homogeneous functions.
For $M$ being a single point, we get a canonical homogeneity structure on the supermanifold $\R^{n|m}$. Homogeneity supermanifolds of this type we call \emph{linear supermanifolds}.

If $F=N\ti\R^{p|q}$ is another trivial vector superbundle, with projection $\zp:F\to N$, then VB-morphisms $\zF:E\to F$ are just morphisms of these homogeneity manifolds. They automatically induce morphisms $\zf:M\to N$ of supermanifolds such that the diagram
$$
\xymatrix@C+20pt@R+10pt{
E \ar[r]^{\zF}\ar[d]_{\zt} & F\ar[d]^{\zp} \\
M \ar[r]^{\zf} & N}
$$
is commutative.

\mn Now, a general \emph{vector superbundle} can be described as a fiber bundle $\zt:E\to M$, equipped with an atlas of local trivializations $(U_a,\zf_a)$,
$$\zf_a:V_a=\zt^{-1}(U_a)\to U_a\ti\R^{n|m},$$
such that the transition maps
$$\zf_b\circ\zf_a^{-1}:(U_b\cap U_a)\ti\R^{n|m}\to(U_b\cap U_a)\ti\R^{n|m}$$
are automorphisms of the homogeneity manifolds $(U_b\cap U_a)\ti\R^{n|m}$. Consequently, the local weight vector fields $\n_a$ give rise to a globally defined weight vector field $\n_E$ on $E$, called the \emph{Euler vector field}, so $E$ is canonically a homogeneity manifold, with $M$ embedded as a submanifold being the zero-locus
of all 1-homogeneous functions.

\mn Note that the weight vector field $\n$ is in this case complete, and the corresponding flow $\zf_t$ of automorphisms of the homogeneity manifold $E$ is the multiplication by $s=e^t\in\R_+$ in $\R^{n|m}$. An important observation is that this $\R_+$-action on $E$ extends canonically to a smooth action of the monoid $(\R,\cdot)$ of multiplicative reals. This is exactly the effective way of defining vector (super)bundles \cite{Grabowski:2012,Jozwikowski:2016} without any reference to the addition in fibers.
\end{example}

\begin{example} \textbf{(Homogeneity spheres)}
Consider $\R^{n+1}$ with the standard coordinates $(x_0,x_1,\dots,x_n)$, and the 1-sphere $S^n\subset\R^{n+1}$ defined as the submanifold of points satisfying the equation $\sum_{i=0}^nx_i^2=1$. Let us distinguish on $S^n$ two charts, $\zf_i:U_i\to\R^n$, where $U_i$ consists of points of $S^n$ such that $x_0=(-1)^i$, $i=0,1$, associated with the stereographic projections from the north $(1,0,\dots,0)$  and the south pole $(-1,0,\dots,0)$,
$$X_k=[\zf_0(x)]_k=\frac{x_k}{1-x_0}\,,\quad Y_k=[\zf_1(x)]_k=\frac{x_k}{1+x_0}\,,\ k+1,\dots,n\,.$$
It is easy to see that the transition map is
$$\zf_{10}=\zf_1\circ\zf_0^{-1}\,:\,(\R^n)^\ti=\R^n\setminus\{ 0\}\to (\R^n)^\ti=\R^n\setminus\{ 0\}\,,\quad
Y=\zf_{10}(X)=\frac{1}{\sum_{k=1}^nX_k^2}\,X\,.$$
In other words, $Y_k=X_k/s^2$, where $s^2=\sum_{k=1}^nX_k^2$. Suppose now we declare the weight 1 for all coordinates $X_k$ and the weight $-1$ for all coordinates $Y_k$. In that case, the transition map preserves the weights and the weight vector field $\nabla$ on $S^n$, which is $\nabla=\sum_kX_k\pa_{X_k}$ in coordinates $(X_k)$ and $\nabla=-\sum_kY_k\pa_{Y_k}$ in coordinates $(Y_k)$, is well defined and introduces a homogeneity structure on $S^n$, which is therefore a compact homogeneity manifold.

\mn A particular case is that of a circle. With respect to the standard parametrization of the unit circle in $\R^2$, $t\mapsto\big(\cos(t),\sin(t)\big)$, and the poles being $p_\mp=(\pm 1,0)$, the above weight vector field $\n$ on $S^1$ reads $\n=\sin(t)\pa_t$. On $S^1\setminus\{p_+\}$, the coordinate change $t=2\arctan(1/x)$ makes $\n$ equivalent to the vector field $x\pa_x$ on $\R$, where homogeneous functions are homogeneous polynomials having weights $\le 0$. On the other hand, on $S^1\setminus\{p_-\}$ the vector field $\n$ is equivalent to $-y\pa_y$ on $\R$, where homogeneous functions are homogeneous polynomials having weights $\le 0$. Hence, global homogeneous functions on $S^1$ are only constants.
\end{example}

\begin{example}\textbf{(Graded bundles)} The \emph{graded bundles} of Grabowski and Rotkiewicz \cite{Grabowski:2012} (see also \cite{Bruce:2015,Bruce:2015a,Bruce:2016,Bruce:2017a}) are purely even manifolds
with an action of the monoid of multiplicative reals, $h_t:F\to F$, with $h_1=\id_F$ and $h_t\circ h_s=h_{ts}$. Somehow unexpectedly, it leads to a structure of a homogeneity manifold on $F$ as follows. It is easily seen that $h_0$ is a surjective submersion of $F$ onto the submanifold $M=h_0(F)$.  One proves that it is, in fact, a fiber bundle with the typical fiber $\R^p$ for some $p$ and there is an atlas of local trivializations $h_0^{-1}(U)\simeq U\ti\R^p$, with coordinates $(x^a)$ on $M$ and linear coordinates $y^i$ on $\R^p$, such that the $\R$-action looks like
$$h_t(x^a,y^i)=(x^a,t^{w_i}y^i)\,,$$
where $w_i$ are non-negative integers. The weight vector field $\nabla_F$ is the generator of the one-parameter group of diffeomorphisms $h_t$, $t>0$. Clearly, $\nabla=\sum_iw_i\cdot y^i\pa_{y^i}$. The transition functions are automatically polynomial in $y^i$, usually non-linear, so $F$ is generally not a vector bundle. This is therefore a purely even version of a graded manifold of Voronov \cite{Voronov:2002}, however, defined in terms of an action of $(\R,\cdot)$ and not with homogeneous coordinates from the very beginning. One can introduce analogously also N-manifolds (see \cite{Jozwikowski:2016}), as supermanifolds equipped with an $(\R,\cdot)$-action $h$ such that $h_{-1}$ acts as the parity operator.
\end{example}

\begin{example} \textbf{(Homogeneity superspheres)}
Consider $M=\R^{3|2}$ with coordinates $(x_0,x_1,x_2,\zx,\zh)$, and the submanifold $S^{2|2}$ in $M$ defined by the equation
$$x_0^2+x_1^2+x_2^2-\zx\cdot\zh=1$$
(the supersphere). Similarly to above, define two open submanifolds $U_\pm$ in $S^{2|2}$ defined by $x_0\ne\pm 1$. We can consider $U_\pm$ as charts with coordinates $\zf_\pm:U_\pm\to\R^{2|2}$ given by
$$\phi_\pm((x_0,x_1,x_2,\zx,\zh)=(X_1^\pm,X_2^\pm,\zx^\pm,\zh^\pm)\,,$$
where
$$X_i^\pm=\frac{x_i}{\sqrt{1+\zx\,\zh}\pm x_0}\quad \text{and}\quad (\zx^\pm,\zh^\pm)=(\zx,\zh)\,,\ i=1,2\,.$$
By direct calculations, one gets
$$X_i^+=\frac{X_i^-}{(X_1^-)^2+(X_2^-)^2}\,,\ i=1,2\,,$$
so the weight vector fields $\nabla^\pm$ on $U^\pm$ read
$$\nabla^\pm=\mp\left(X^\pm_1\pa_{X^\pm_1}+X^\pm_2\pa_{X^\pm_2}\right)+w\,\zx^\pm+v\,\zh^\pm\,,$$
where $w,v$ are arbitrary real numbers, giving rise to a globally defined weight vector field $\nabla$ on  the supersphere $S^{2|2}$.
\end{example}

\begin{example}\label{ex1} \textbf{(Linear homogeneity supermanifolds)}
Consider the supermanifold $M=\R^{k|l}$ with the standard global coordinates $(y^a,\zh^j)$, where $y^a$ are even and $\zh^j$ are odd, and a multi-index
$$\zm=(w_1,\dots,w_k,v_1,\dots,v_l)\in\R^{k}\ti\R^l\,.$$
The weight vector field on $M$,
\be\label{zl}\nabla_\zm=\sum_{a=1}^kw_a\cdot y^a\,\pa_{y^a}+\sum_{j=1}^lv_j\cdot\zh^j\,\pa_{\zh^j}\,,\ee
defines a \emph{linear homogeneity structure of type $\zm$} on $\R^{k|l}$, denoted $\R^{k|l}(\zm)$. If all weights $(w_a,v_j)$ are positive, we call the linear homogeneity supermanifold \emph{positive}.
\end{example}

\begin{example}\label{hps} \textbf{(Projective homogeneity manifold)} Take a linear homogeneity manifold $M=\R^{n|m}(\zm)$, coordinates $(y^i,\zx^a)$, $n>0$, and $\zm=(w_1,\dots,w_n,v_1,\dots,v_m)$ (cf. Example \ref{ex1}).
Let $(\R^{n|m}(\zm))^\ti$ be the open submanifold of $\R^{n|m}(\zm)$ characterized by  $y\ne 0$.
We define the \emph{projective homogeneity supermanifold} $\mathbb{P}\R^{(n-1)|m}(\zm)$ as the manifold of cosets of $(\R^{n|m}(\zm))^\ti$ with respect to the equivalence relation
\be\label{id}(y^1,\dots,y^n,\zx^1,\dots,\zx^m)\sim(t\cdot y^1,\dots,t\cdot y^n,t\cdot\zx^1,\dots,t\cdot\zx^m)\,,\ee
where $t\ne 0$. This space can also be understood as the set of orbits of the obvious action of the group $\R^\ti$ of multiplicative reals on $(\R^{n|m}(\zm))^\ti$, so that the canonical projection
\be\label{proj}\zt:\R^{n|m}(\zm)^\ti\to\mathbb{P}\R^{(n-1)|m}(\zm)\ee carries a structure of an $\R^\ti$-principal bundle.

Let us denote the class of coordinates $(y^i,\zx^a)$ with respect to the identification (\ref{id}) with $[(y^i,\zx^a)]$. The projective space $\mathbb{P}\R^{n|m}(\zm)$ is canonically a homogeneity supermanifold, covered by coordinate neighbourhoods $U_i=\{[(y,\zx)]\,|\,y^i\ne 0\}$, $i=1,\dots,n$,
with coordinates  $\zf_i:U_i\to\R^{(n-1)|m}(\zm_i)$,
$$\zf_i:[(y^1,\dots,y^n,\zx^1,\dots,\zx^m)]\mapsto\left(z^1_i=\frac{y^1}{y^i},\dots,\widehat{y^i},\dots,
z^{n}_i=\frac{y^n}{y^i},\zvy^1_i=\frac{\zx^1}{y^i},\dots,\zvy^m_i=\frac{\zx^m}{y^i}\right)\,,$$
where `$\widehat{\ \ }$' stands for omission. Here, the coordinates $(z_i^1,\dots,z_i^{i-1},z_i^{i+1},\dots,z_i^n)$ and $(\zvy_i^1,\dots,\zvy_i^m)$ have weights
$$(w_1-w_i,\dots,w_{i-1}-w_i,w_{i+1}-w_i,\dots,w_n-w_i)\quad\text{and}\quad (v_1-w_i,\dots,v_m-w_i)\,,$$
respectively.
The transition functions are (for $i<j$):
\begin{align}\nn &\zf_j\circ\zf_i^{-1}\Big(z^1_i,\dots,z^{i-1}_i,z^{i+1}_i,\dots,z^n_i,\zvy^1_i,\dots,\zvy^m_i\Big)=\\
&\left(\frac{z^1_i}{z^j_i},\cdots,\frac{1}{z^j_i},\dots,\frac{z^{j-1}_i}{z^j_i},\frac{z^{j+1}_i}{z^j_i},
\dots\frac{z^{n}_i}{z^j_i},\frac{\zvy^1_i}{z^j_i},\dots,\frac{\zvy^m_i}{z^j_i}\right)\,.\nn
\end{align}
It is now clear that the transition functions are smooth and transform homogeneous coordinates into homogeneous coordinates, so $\mathbb{P}\R^{(n-1)|m}(\zm)$ is a  homogeneity supermanifold. It is also easily seen that its body is the projective space $\R\mathbb{P}^{(n-1)}$.
\end{example}
\section{Tangent and cotangent lifts of homogeneity}
In this section, we define canonical homogeneity supermanifold structures on the tangent $\sT M$ and the cotangent $\sT^*M$ bundles of a homogeneity supermanifold $(M,\n)$. In local homogeneous coordinates, we mimic the tangent and phase lifts from the purely even situation \cite{Bruce:2016,Grabowski:2009,Grabowski:2012}.

\begin{theorem}\label{e1} (\textbf{Lifts of homogeneity structures})
Let $M$ be a homogeneity supermanifold with the weight vector field $\nabla_M$, which in homogeneous coordinates reads
$$\nabla_M=\sum_aw_a\cdot x^a\pa_{x^a}\,.$$
Then, $\sT M$ and $\sT^*M$ are canonically homogeneity vector superbundles with the weight vector fields written in the adapted coordinates as
\be\label{tl}\dt\nabla_M=\sum_aw_a\left(x^a\pa_{x^a}+\dot x^a\pa_{\dot x^a}\right)\ee
and
\be\label{cl}\dts\nabla_M=\sum_aw_a\left(x^a\pa_{x^a}-p_a\pa_{p_a}\right)\,.\ee
In particular, if $M$ is a $\zG$-homogeneity supermanifold for some $\zG\subset\R$, then $\sT M$ is also a $\zG$-homogeneity supermanifold, and $\sT^*M$ is a $(\pm\zG)$-homogeneity supermanifold.
\end{theorem}
\begin{proof}
The proof is a direct check that the transition functions in the tangent and the cotangent bundles respect the local forms of the above vector fields.

\end{proof}
\no Let us observe that the lifted weight vector fields $\dt\nabla_M$ and $\dts\nabla_M$ are particular cases of the complete lifts of vector fields on $M$ to the tangent $\sT M$ and cotangent $\sT^*M$ bundle, as defined in the literature in the purely even case (e.g., \cite{Grabowski:1995,Yano:1973}).

\medskip
For the weight vector field $\nabla_M=\sum_aw_a\cdot x^a\pa_{x^a}$ we get
$\dt\nabla_M$ as in Theorem \ref{e1}. It is easily seen that $\dt Y$ is even for even $Y$.

\medskip
An intrinsic construction of $\dt Y$ in \cite{Grabowski:1995} looks as follows.
The vector field $X$ corresponds to a linear function $\zi_X$ on $\sT^*M$. The differential of this function corresponds, in turn, to a function $\zi_{\xd\zi_X}$ on $\sT\sT^*M$ which is linear with respect to both vector bundle structures. But, there is a canonical isomorphism of double vector bundles $\ze_M:\sT^*\sT M\to\sT\sT^*M$, discovered in 70's by Tulczyjew \cite{Tulczyjew:1974}. Hence, $\zi_{\xd\zi_X}\circ \ze$ is a linear function on $\sT^*\sT M$, thus corresponds to a (linear) vector field $\dt Y$ on $\sT M$. There is also an algebroid version of the tangent lift of vector fields \cite{Grabowski:1997,Grabowski:1999}.

The cotangent lift $\dts Y$, in turn, is sometimes defined as the Hamiltonian vector field on $\sT^*M$ with the Hamiltonian being the linear function $H=\zi_Y$ on $\sT^*M$ corresponding to the vector field $Y$. We propose the following simple definitions.
\begin{definition} Let $Y$ be a vector field on a supermanifold $M$. The \emph{tangent lift} $\dt Y$ is a vector field on $\sT M$ uniquely determined by the identity
\be\label{tl1} \dt Y(\zi_\za)=\zi_{\Ll_Y\za}\,,\ee
and the \emph{cotangent lift} $\dts Y$ is a vector field on $\sT^*M$ uniquely determined by the identity
\be\label{cl1} \dts Y(\zi_X)=\zi_{[Y,X]}\,,\ee
where $\za$ is an arbitrary 1-form, and $X$ is an arbitrary vector field on $M$.
\end{definition}
The tangent lift $\dt Y$ of a vector field $Y=\sum_af^a(x)\pa_{x^a}$ on $M$ is given in the adapted local coordinates $(x^a,\dot x^b)$ on $\sT M$ by (see \cite{Grabowski:1995} for a purely even version and cf. (\ref{Ll}))
\be\label{lift1} \dt Y=\sum_a\left(f^a(x)\pa_{x^a}+\left(\sum_b\dot x^b\frac{\pa f^a}{\pa x^b}(x)\right)\,\pa_{\,\dot x^a}\right)\,.\ee

\mn The cotangent lift $\dts Y$ of a vector field $Y=\sum_af^a(x)\pa_{x^a}$ of parity $\zs$ on $M$  is given in adapted local coordinates $(x^a,p_b)$ on $\sT^*M$ by
\be\label{lift2}
\dts Y=\sum_a\left(f^a(x)\pa_{x^a}-(-1)^{\zs\cdot\zs_a}\left(\sum_bp_b\frac{\pa f^b}{\pa x^a}(x)\right)\pa_{p_a}\right)\,.
\ee
The integral version of (\ref{tl1}), valid only for even vector fields, reads (cf. Remark \ref{rw})
$$
\Exp(t\cdot\dt Y)=\sT\Exp(t\,Y)\,.
$$
This means that $\dt Y$ is the generator of the local one-parameter group of diffeomorphisms $\sT\Exp(t\,Y)$.
The integral version of (\ref{cl1}) for even $Y$ reads, in turn,
$$ \Exp(t\,\dts Y)=\left(\sT\Exp(-t\,Y)\right)^*\,,$$
where $\left(\sT\Exp(t\,Y)\right)^*:\sT^*M\to \sT^*M$ is the local automorphism of the vector superbundle $\sT^*M$ which is dual to the local isomorphism $\sT\Exp(t\,Y):\sT M\to \sT M$.

\mn Let us observe that, due to the fact that transition functions for $\sT M$ are linear in coordinates $(\dot x^a)$, the shift of their degrees by $\zl\in\Z_2\ti\R$,  $\deg(\dot x^a)\mapsto\deg(\dot x^a)+\zl$, is consistent, so it defines a new homogeneity supermanifold structure on $\sT M$. This homogeneity supermanifold we denote $\sT[\zl]M$. Similarly, we can proceed with the cotangent bundle and get $\sT^*[\zl]M$.
The linear coordinates in $\sT[\zl]M$ and $\sT^*[\zl]M$ are therefore of degrees $\deg(\dot x^a)=\deg(x^a)+\zl$ and $\deg(p_a)=-\deg(x^a)+\zl$, respectively. The dual of $\sT[\zl]M$ is in a natural sense $\sT^*[-\zl]M$, so the canonical pairing,
$$\la\cdot,\cdot\ran:\sT[\zl]M\ti_M\sT^*[-\zl]M\to\R,$$
is a morphism of homogeneity supermanifolds.
Note that $\sT[(1,0)]M$ and $\sT^*[(1,0)]M$ are often denoted with $\zP\sT M$ and $\zP\sT^*M$.
The tangent and cotangent lifts of $\n_M$ to $\sT[\zl]M$ and $\sT^*[\zl]M$ have formally the same form as  (\ref{tl}) and (\ref{cl}), only the degrees of $\dot x^a$ and $p_a$ are shifted.

\mn Vector fields $X$ and one-forms $\za$ on $M$, that are homogeneous of degree $\zl$, induce morphisms of the homogeneity manifolds $X^\sharp:M\to\sT[\zl]M$ and $\za^\sharp:M\to\sT^*[\zl]M$, usually understood as sections. In local homogeneous coordinates, in which $X=\sum_af^a(x)\,\pa_{x^a}$ and $\za=\sum_a\xd x^a\,g_a(x)$, this sections viewed as maps look like
$$X^\sharp(x^a)=\big(x^a,f^b(x)\big)\,,\quad\text{and}\quad\za^\sharp(x^a)=\big(x^a,g_b(x)\big)\,.$$
It is easy to check that they are indeed morphisms of homogeneity supermanifolds. Note also that the formulae (\ref{tl}) and (\ref{cl}) define also weight vector fields on $\sT[\zl]M$ and $\sT^*[\zl]M$ for each $\zl\in\Z_2\ti\R$.
\section{Homogeneous distributions and codistributions}
Let $M$ be a supermanifold of total dimension $n$. By a \emph{distribution} $D$ of corank $k$ on $M$ (and rank $(n-k)$) we will understand a vector super subbundle in $\sT M$ of corank $k$, i.e., a vector subbundle spanned locally by $(n-k)$ nowhere-vanishing vector fields $X_1,\dots,X_{n-k}$. These vector fields are therefore linearly independent in the sense that they form a local basis of a locally free module over $\cA(M)$.

An equivalent definition is that $D$ is a submanifold in $\sT M$ determined locally by the system of equations $\zi_{\za_1}=0,\zi_{\za_2}=0,\dots,\zi_{\za_k}=0$, where $(\za_i)$ is a system of locally nowhere-vanishing linear ($1$-homogeneous with respect to the Euler vector field) functions, i.e., nowhere-vanishing differential one-forms which are linearly independent in the above sense. These one-forms span therefore a vector subbundle in $\sT^*M$ (\emph{codistribution}) of rank $k$ (and corank $(n-k)$) which we call the \emph{annihilator} of $D$ and denote $D^o$. It is the submanifold in $\sT^*M$ defined by equations $\zi_{X_1}=0,\dots,\zi_{X_{n-k}}=0$. It is clear that $(D^o)^o=D$ in the obvious sense.

Let us observe that it makes sense to speak about even or odd distributions (codistributions) as generated locally by even or odd vector fields (one-forms). Indeed, suppose a rank $k$ distribution is locally freely generated by even vector fields, that in local coordinates $(x^a)$ on $M$ read $X_i=\sum_af^a_i(x)\pa_{x^a}$, $i=1,\dots,k$. We will show that there are no free generators of $D$ containing an odd vector field, say, $Y$. Being one of the free generators, $Y$ is nowhere vanishing and can be written in the form $Y=\sum_ig^i(x)\,X_i$. As $Y$ is odd and $X_i$ even, the functions $g^i$ must be odd, so they all vanish on the submanifold $|M|$; a contradiction. If $X_i$ are odd and $Y$ is even, then again $g^i$ must be odd, so vanishing on $|M|$. Analogously for codistributions. This makes the following definition meaningful.
\begin{definition}
A distribution $D$ we call \emph{even} (resp., \emph{odd}) if it is locally freely generated by even (resp, odd) vector fields. Similarly, we define even and odd codistributions.
\end{definition}
\no Let $D$ be a distribution on a supermanifold $M$. It is easy to see that we have
\be\label{ann1} X\in D \ \Leftrightarrow\ \zi_X\big|_{D^o}=0\quad\text{and}\quad \za\in D^o\ \Leftrightarrow\ \zi_\za\big|_D=0\,.
\ee
These identities we will use in the proof of the following proposition.
\begin{proposition}
Let $Y=\sum_af^a(x)\pa_{x^a}$ be an even vector field on a supermanifold $M$ of dimension $r$, and let $D$ be a rank $k$ distribution on $M$. Then, the following are equivalent:
\begin{description}
\item{(1)} The tangent lift $\dt Y$ is tangent to the submanifold $D\subset\sT M$.
\item{(2)} $D$ is invariant with respect to the local one-parameter group of diffeomorphisms $\Exp(t\,\dt Y)$ of $\sT M$.
\item{(3)} The cotangent lift $\dts Y$ is tangent to the annihilator vector subbundle $D^o$ in $\sT^*M$ (see Lemma \ref{l2}).
\item{(4)} $D^o$ is invariant with respect to the local one-parameter group of diffeomorphisms $\Exp(t\,\dts Y)$ of $\sT^*M$ (see Lemma \ref{l2}).
\item{(5)} If $X_1,\dots,X_k$ are local vector fields on $M$ locally generating $D$, then
$$\Ll_Y(X_i)=[Y,X_i]=\sum_{j=1}^k f^j_i(x)X_j$$ for some functions $f^j_i$ on $M$. In other words, $\Ll_Y(X)\in D$ if $X\in D$.
\item{(6)} If $\za^1,\dots,\za^{(r-k)}$ are local one-forms on $M$ which locally generate $D^o$, then
$$\Ll_Y(\za^i)=\sum_{j=1}^{(r-k)} g_j^i(x)\za^j$$ for some functions $g_j^i$ on $M$. In other words,
$\Ll_Y(\za)\in D^o$ if $\za\in D^o$.
\end{description}
\end{proposition}
\begin{proof}
The equivalences $(1)\,\Leftrightarrow\,(2)$ and $(3)\,\Leftrightarrow\,(4)$ are standard.
As $\Exp(t\,\dt Y)=\sT\Exp(t\,Y)$, the local linear diffeomorphism $\sT\Exp(t\,Y)$ preserves $D$. But this is equivalent to the fact that its dual $(\sT\Exp(t\,Y))^*=\Exp(-t\,\dts Y)$ preserves $D^o$, so $(2)\ \Leftrightarrow\ (4)$.
We can easily prove also $(5)\,\Leftrightarrow\,(6)$. As we have $\la\za^i,X_j\ran=0$,
$$0=\Ll_Y\la\za^i,X_j\ran=\la\Ll_Y(\za^i),X_j\ran+\la\za^i,\Ll_Y(X_j)\ran\,.$$
Then, $\la\za^i,\Ll_Y(X_j)\ran=0$ if and only if  $\la\Ll_Y(\za^i),X_j\ran=0$.

To show $(1)\,\Rightarrow\,(6)$, take a one-form $\za\in D^o$. According to (\ref{ann1}), this implies $\zi_\za\big|_D=0$. By (1), $\dt Y$ is tangent to $D$, so that $[\dt Y(\zi_\za)]\big|_D=0$. But $\dt Y(\zi_\za)=\zi_{\Ll_Y(\za)}$, so, again by (\ref{ann1}), $\zi_{\Ll_Y(\za)}\big|_D=0$, i.e., $\Ll_Y\za\in D^o$. It remains to show $(6)\ \Rightarrow\ (1)$. A vector field is tangent to the submanifold $D\subset\sT M$ if and only if it differentiates local smooth functions vanishing on $D$ into local smooth functions vanishing on $D$. From the previous reasoning, we know that $\zi_\za$ vanish on $D$ for all $\za\in D^o$ and that $\dt Y(\zi_\za)\big|_D=0$. It is easy to see that the family of functions $\{\zi_\za\,|\,\za\in D^o\}$ is reach enough to assure that $\dt Y$ is tangent to $D$.

\end{proof}
\no Now, we will pass to distributions on homogeneity supermanifolds.
\begin{definition} Suppose that $D$ is a distribution of rank $(n-k)$ (resp., a codistribution of rank $k$) on a homogeneity supermanifold $M$ of total dimension $n$ and weight vector field $\nabla_M$. We call $D$ a \emph{homogeneous distribution} (\emph{homogeneous codistribution}) if the tangent lift $\dt\nabla_M$ (resp., the cotangent lift $\dts\nabla_M$) is tangent to $D$.
\end{definition}
\no Note that there is no relevant concept of the degree of a homogeneous distribution (codistribution) $D$, as shown in the following example.
\begin{example}
Let us consider $\R^3$ with the standard coordinates $(x,y,z)$ and the weight vector field $\nabla=\pa_x+\pa_y$.
It is indeed a weight vector fields, because for $w\in\R$, $w\ne 0$, the new coordinates $(x',y',z')=(e^{w\cdot x},e^{-w\cdot y},z)$ are global homogeneous coordinates with $\deg(x')=(0,w)$, $\deg(y')=(0,-w)$, and $\deg(z')=(0,0)$. It is clear that the one-form on $\R^3$,
$$\za=\xd z'+x'\,\xd y'=\xd z+e^{w\cdot x}\,\xd e^{-w\cdot y}\,,$$ has weight 0 and is nowhere vanishing, thus generates a rank 1 codistribution $[\za]$. Hence,
$\za'=e^{\zb x}\,\za$ is a one-form of weight $\zb$. Moreover, $\za'$ is another generator of $[\za]$.
\end{example}

\begin{theorem}\label{homdist}
Let $D\subset\sT M$ be a distribution on a homogeneity manifold $(M, \nabla)$. Then, $D$ is  homogeneous if and only if $D$ is locally generated by homogeneous vector fields.
\end{theorem}
\begin{proof}
As $\dt \nabla$ is a linear vector field on $\sT M$, it maps basic functions into basic functions, and the part `only if' is obvious.
Suppose now that the total dimension of $M$ is $n$, and $D$ is homogeneous of rank $k$, so that $[\n, X]\in D$ for every $X\in D$, and let $m\in |M|$.

\medskip\noindent Suppose now that $\n(m)\ne 0$. Let $D$ be trivializable over a neighbourhood $U$ of $m$, and let $S$ be a submanifold of $U$ transversal to $\n$ such that $m\in |S|$, and $X_i$, $i=1,\dots,k$, be local free generators of $D$ over $S$. If $\varphi^t$ is a local flow of $\n$ in a neighbourhood of $m$, then $\sT\varphi^t$ is a local flow of $\dt\n$, and there are unique local extensions $\hat X_i$ of $X_i$, $i=1,\dots,k$, which are $\sT\varphi^t$-invariant, $(\sT\varphi^t)_*(\hat X_i)=\hat X_i$. Since $D$ is $\sT\varphi^t$-invariant ($\dt \nabla$ is tangent to $D$), $\hat X^i$, $i=1,\dots,k$, form a local basis of generators of $D$ in a neighbourhood of $m$. Of course, $\hat X_i$ are $\n$-homogeneous of degree 0, $[\n,\hat X_i]=0$.

\medskip\noindent
If, in turn, $\n(m)=0$, then we have locally homogeneous coordinates $(x^a)$ in a neighbourhood of $x_0$, so that
$$\n=\sum_iw_ix^i\pa_{x^i},$$
for some $w_i\in\R$, and $\sT_{m}M$ is $\dt\n$-invariant, thus a graded vector space spanned by homogeneous vectors $\pa_{x^i}$. As $D_{m}$ is a graded subspace, it is also spanned by some of these homogeneous vectors, say, $\pa_{x^i}$ with $i=1,\dots,k$.

Consider a local basis $\hat X_i$, $i=1,\dots,k$, of $D$ such that $\hat X_i(m)=\pa_{x^i}$. Since sections of $D$ form a $C^\infty(M)$-module, we can additionally assume that
\be\label{gener}\hat X_i=\pa_{x^i}+\sum_{j=k+1}^n f^j_i(x)\pa_{x^j},\ee
where $f^j_i(m)=0$. In other words, $\hat X_i=\pa_{x^i}+Y_i$, where $Y_i$ is a linear combination of $\pa_{x^j}$, $j>k$, and $Y_i(m)=0$. As $D$ is $\dt\n$-invariant, we have
$$[\n,\hat X_i]=-w_i\pa_{x^i}+[\n,Y_i],$$
where, again, $[\n,Y_i]$ is a linear combination of $\pa_{x^j}$, $j>k$.
On the other hand, $[\n,X_i]$ is a linear combination of $X_l$, $l=1,\dots,k$, so
$$[\n,\hat X_i]=-w_i\hat X_i,$$
and $\hat X_i$, $i=1,\dots,k$, are homogeneous.

\end{proof}
\no A full analog of the Frobenius Theorem for supermanifolds is well known (see, e.g., \cite[Theorem 6.3]{Tuynman:1983} or \cite{Shander:1983}).
We will now prove its homogeneous version.
\begin{theorem}[Homogeneous Frobenius Theorem]\label{frob}
Let $D\subset\sT M$ be a rank-$k$ involutive homogeneous distribution on a homogeneity supermanifold $(M, \nabla)$ of total dimension $n$, and let $m\in|M|$. Then, there is a neighbourhood of $m$ endowed with a system $(x^i)$, $i=1,\dots,n$, of homogeneous local coordinates such that
\begin{description}
\item{(a)} $D$ is locally spanned by $\la\pa_{x^1},\dots,\pa_{x^k}\ran$, if $\n(m)=0$;
\item{(b)} $\n=x^n\pa_{x^n}$, and $D$ is locally spanned by  $\la\pa_{x^1},\dots,\pa_{x^k}\ran$, or by  $\la\pa_{x^1},\dots,\pa_{x^{k-1}},Y\ran$, where
$$Y=\n+\sum_{j=k}^{n-1}h_j(x^k,\dots,x^{n-1})\pa_{x^j},$$
in the case $\n(m)\ne 0$.
\end{description}
\end{theorem}
\begin{proof}
Suppose first that $\n(m)=0$, Like in the proof of Theorem \ref{homdist}, we can choose local homogeneous coordinates $(x^i)$ around $m$ such that $D$ is locally generated by the vector fields
(\ref{gener}). If $D$ is involutive, $[X_i,X_j]\in D$ for $i,j=1,\dots,k$. But the latter brackets are linear combinations of $\pa_{x^l}$, $l>k$, which belong to $D$, so they are linear combinations of $X_s$, $s=1,\dots,k$. Consequently, $[X_i,X_j]=0$. In particular, those $X_i$ which are odd are integrable. Now, we can start with coordinates $(x^j)$, $j=k+1,\dots,n$ and follow the proof of \cite[Proposition 6.1]{Tuynman:1983} using commuting flows of the vector fields $X_i$, $i=1,\dots,k$ (flows with an odd `time' in the case of integrable odd vector fields) to complete the variables $(x^j)$, $j=k+1,\dots,n$ with variables $y^i$, $i=1,\dots,k$, vanishing at $m$ such that $X_i=\pa_{y^i}$. Obviously, $y^i$ are homogeneous. The existence of flows follows, e.g., from the `straightening theorem' \cite{Shander:1980}.

\medskip\noindent If, in turn, $\n(m)\ne 0$, we may choose a neighbourhood $U$ of $m$ and local coordinates $(x,z^1,\dots,z^{n-1})$ in $U$ such that $x$ is even, $x(m)=1$, $z^j(m)=0$, and $\nabla=x\,\pa_x$ (Corollary \ref{cor}). We can assume additionally that $D$ is trivializable over $U$, and $X_i$, $i=1,\dots,k$, are local free generators of $D$ over $U$. We know from the proof of Theorem \ref{homdist} that we can choose generators $X_i$ to be of degree 0.
Let $S$ be a 1-codimensional submanifold of $U$ defined by the condition $x=1$. If all $X_i(m)$ are tangent to $S$, then we can choose $z^j$ such that $X_i(m)=\pa_{z^i}$, $i=1,\dots,k$, and we can proceed as in the previous case.

\mn If there is $X_i(m)$ not tangent to $S$ at $m$, say, $X_k(m)\notin\sT_mS$, then
$$X_k=f(x,z)\n+\sum_{j=1}^{n-1}g_j(x,z)\pa_{z^j},$$
for some functions $f,g_j$, $f(m)\ne 0$. As $[\n,X_k]=0$ ($X_k$ is of degree $0$), then $f,g_j$ do not depend on the variable $x$. Hence, the vector field
\be\label{dd}\hat X_k=\n+\sum_{j=1}^{n-1}h_j(z)\pa_{z^j},\ee
where $h_j=g_j/f$ is of degree $0$ and belongs to $D$.
Now, we can use $X_k$ to kill the parts containing $\n$ in $X_1,\dots,X_{k-1}$, and reduce the situation to the case when $X_i$, $i<k$, depend on coordinates $(z^j)$ only. They are of degree $0$ and span an involutive distribution of rank $(k-1)$. This is a pure supermanifold case (the degrees of $z^j$ are $0$), so we can assume that this distribution is generated by $\la\pa_{z^1},\dots,\pa_{z^{k-1}}\ran$. Consequently, $D$ is spanned by $\la\pa_{z^1},\dots,\pa_{z^{k-1}},\hat X_k\ran$, where $\hat X_k$ is of the form (\ref{dd}). Moreover, we can have $h_j=0$ for $j=1,\dots,k-1$, and, from involutivity, that $h_j$ do not depend on $z^a$ for $a<k$. In other words, we can view $Y=\hat X_k$ as an even vector field of degree $0$
in coordinates $(x,z^k,\dots,z^{n-1})$,
$$Y=x\pa_x+\sum_{j=k}^{n-1}h_j(z^k,\dots,z^{n-1})\pa_{z^j}.$$

\end{proof}
\section{Homogeneity vector superbundles}
Suppose that on a supermanifold $M$ we have two weight vector fields $\n$ and $\n'$. A natural definition of their compatibility is the following.
\begin{definition}
Weight vector fields on a supermanifold $M$ we call \emph{compatible} if there exists an atlas on $M$ with local coordinates being \emph{bi-homogeneous}, i.e., homogeneous with respect to both homogeneity structures. In this case, we call the triple $(M,\n,\n')$ a \emph{double homogeneity manifold}. If a function $f$ on $E$ is homogeneous of degree $w$ with respect to $\n$ and degree $w'$ with respect to $\n'$, then we call $f$ \emph{homogeneous of bi-degree $(w,w')$}.

\mn A \emph{homogeneity vector superbundle} is a double homogeneity manifold $(E,\n,\n')$ in which $\n'$ is the Euler vector field of a vector bundle structure $\zt:E\to M$ on $E$.
\end{definition}
Our approach to homogeneity vector superbundles is natural and simple; we do not need to use cocycles with values in supergroups, etc. Moreover, our approach is universal and works well also for other classes of graded manifolds, e.g., for $\Z^n_2$-manifolds \cite{Bruce:2025}. Like in the case of \emph{double vector bundles} \cite{Grabowski:2009}, one can expect that the compatibility condition for the homogeneity vector superbundle can be characterized simply by the commutation of the weight vector fields. This is an open question.
\begin{conjecture}\label{conjecture} A vector bundle structure on a homogeneity manifolds $(E,\n)$ is compatible if and only if $[\n,\n_E]=0$, where $\n_E$ is the Euler vector field.
\end{conjecture}
\no For any homogeneity vector superbundle $(E,\n,\n')$, where $\n'$ is the Euler vector field of a vector bundle  structure $\zt:E\to M$ on $E$, the weight vector field $\n$ is linear on $E$, so projectable and tangent to $M$, thus turning $M$ into a homogeneity supermanifold.
\begin{example}
A weight vector field $\n$ on a supermanifold $M$, written locally as
$$\nabla=\sum_aw_ax^a\pa_{x^a}\,,$$
determines a homogeneity structure on $\sT M$, given by the tangent lift
$$\dt\nabla=\sum_aw_a\left(x^a\pa_{x^a}+\dot x^a\pa_{\dot x^a}\right).$$
It turns $\sT M$ into a homogeneity vector superbundle. Bi-homogeneous coordinates are $x^a$ (with weight $(w_a,0)$) and $\dot x^a$ (with weight $(w_a,1$). Of course, the second coefficient is the weight with respect to the Euler vector field $\n_{\sT M}$.

Similarly, the cotangent bundle $\sT^*M$ is canonically a homogeneity vector superbundle with respect to the cotangent lift $\dts\n$ and the Euler vector field $\n_{\sT^* M}$.
\end{example}
\begin{example}\label{hvsT}
Suppose $(E,\n,\n')$ is a homogeneity vector superbundle $\zt:E\to M$ modeled on $\R^{k|l}(\zm)$, the linear homogeneity supermanifold of type $\zm$,  for some multi-index
$$\zm=(w_1,\dots,w_k,v_1,\dots,v_l)\in\R^{k}\ti\R^l\,,$$
i.e., we have local trivializations of $\zt$ of the form
$$\zt^{-1}(U)=U\ti\R^{k|l}(\zm).$$
If the transition maps are linear, i.e., respect linearity of functions, and respect the linear homogeneity structure on $\R^{k|l}(\zm)$ (cf. (\ref{zl})), represented by the weight vector field
$$\nabla_\zm=\sum_{a=1}^kw_a\cdot y^a\,\pa_{y^a}+\sum_{j=1}^lv_j\cdot\zh^j\,\pa_{\zh^j}\,,$$
on $\R^{k|l}$ with the standard coordinates $(y^a,\zh^i)$ (understood as a vector field on the product manifold $U\ti\R^{k|l}(\zm)$), then we get a homogeneity vector superbundle called a \emph{homogeneity vector superbundle of type $\zm$}. In this picture, homogeneous coordinates of degree $w$ in $M$ are bi-homogeneous of bi-degree $(w,0)$, even fiber coordinates $y^a$ are bi-homogeneous of bi-degree $(w_a,1)$, and odd fiber coordinates $\zh^i$ are bi-homogeneous of bi-degree $(v_i,1)$.
\end{example}
\no Note that being of type $\zm$ is a rather strong assumption for a homogeneity vector superbundle. For instance,
the tangent bundle in Example \ref{hvsT} is generally not of this type. This is because there is nothing like a vertical part of $\dt\n$. The local vertical part of $\dt\n$ depends on local homogeneous coordinates on $M$ and changes under transition maps. If $\n$ is not vanishing at a certain $x_0\in|M|$, then we can find local coordinates around $x_0$ of almost arbitrary weights, so the weights of the local vertical part can be completely different.

On the other hand, this type of graded manifold is very common, and almost standard in the case of $\Z$-graded manifolds, where one usually assumes that $M$ is purely even and $w_i$ and $v_j$ are integers. For instance, this is the case of graded manifolds of Mehta \cite{Mehta:2006}, and Voronov \cite{Voronov:2002}.

\mn An useful observation is that on homogeneity superbundles with a typical fiber $F=\R^{k|l}(\zm)$ and the trivial homogeneity structure on the base, we have a well-defined action $h:\R_+\ti E\to E$ of the multiplicative group $\R_+$ of positive reals. This action in homogeneous local coordinates takes the form
\be\label{st}h_t(x^\za,y^a,\zh^j)=(x^\za,t^{w_a}\cdot y^a,t^{v_j}\cdot\zh^j)\,,\ t>0\,.\ee
Note that $h_t\circ h_s=h_{ts}$, so that $(h_t)_{t>0}$ is a one-parameter group of (global) diffeomorphisms of $E$.

Moreover, it is clear that each $h_t$ preserves the weights of homogeneous functions and preserves the form of allowed charts, so that it is an automorphism of the homogeneity superbundle $E$. The $\n_\zm$-weight of a homogeneous function $f$ can be recognized as $w\in\R$ such that $f\circ h_t=t^w\cdot f$ for $t>0$. Similarly, for a $w$-homogeneous vector field $X$, we have $(h_t)_*(X)=t^w\cdot X$, and for a $w$-homogeneous differential form $\zw$ we have $(h_t)^*(\zw)=t^w\cdot\zw$. Diffeomorphisms $h_t$ of this action, $t\ne 0$, are called in \cite{Voronov:2002} \emph{scaling transformations}. For $\N$-graded superbundles, these scaling transformations are actually defined for all $t\in\R$ and the bundle projection $\zt$ is $h_0$.

A useful comparison of various approaches to vector bundles on graded manifolds (sheaves of graded modules, projective graded modules, fiber bundles with fibers having a suitable linear structure) can be found in \cite{Smolka:2025}.

\begin{remark} We can weaken our assumptions for a homogeneity vector superbundle of type $\zm$, and start not with a vector superbundle of rank $(k|l)$ but with an arbitrary fiber bundle $\zt:E\to M$, with a typical fiber $F$ carrying a homogeneity structure $\n_F$, and assume that the transition maps for local trivializations respect the product homogeneity structure, i.e.,
$$t_{ij}:U_{ij}\ti F\to U_{ij}\ti F$$
are morphisms of the corresponding homogeneity manifolds respecting $\n_F$. Such structures we will call \emph{homogeneity superbundles}.
\end{remark}
\section{Cartan calculus}
\mn Let $M$ be a $\zG$-homogeneity supermanifold with local homogeneous coordinates $(x^a)$, $\deg(x^a)=(\zs_a,w_a)$, and let $m\in |M|$. If $X$ is a vector field on $M$, $X=\sum_af_a(x)\pa_{x^a}$, then we denote with $X(m)\in\sT_mM$ the vector $X(m)=\sum_a\wt{f_a}(m)\pa_{x^a}$, where $\wt{f}$ is the restriction of $f$ to the submanifold $|M|$; we just kill all odd coordinates. Similarly, we define $\zw(m)$ for a (local) differential 1-form $\zw$ as an element of $\sT^*_mM$.
We say that a vector field $X$ (a differential form $\zw$) on $M$ is \emph{non-zero} at $m\in |M|$ if $X(m)\ne 0$ (resp., $\zw(m)\ne 0$), and that $X$ ($\zw$) is \emph{nowhere-vanishing} if $X(m)\ne 0$ ($\zw(m)\ne 0$) for all $m\in|M|$. It is easy to see that a one-form is nowhere-vanishing if and only if there is a vector field $X$ such that $i_X\za=1$. The vector field $X$ is then also nowhere-vanishing.

For $\zG\subset\R$, a \emph{$\zG$-graded vector superspace} is a graded vector space $V=\op_{\zg\in\Z_2\ti\zG}V^\zg$. The tangent space $\sT_mM$ is a $(-\zG)$-graded vector superspace consisting of all $X(m)$, i.e., it is spanned by vectors $\pa_{x^a}(m)$. This means that If $x^a$ is of degree $(\zs_a,w_a)$, then $\pa_{x^a}(m)$ is of degree $(\zs_a,-w_a)$.
The 1-form $\xd x^a$ a linear function on $\sT_mM$ of degree $(\zs_a,w_a)$.

\mn All this implies that with every system of homogeneous local coordinates $(x^a)$ on a homogeneity supermanifold $M$ we can associate the adapted coordinate systems in the tangent and cotangent bundles: $(x^a,\dot x^b)$ on $\sT M$ and $(x^a,p_b)$ on $\sT^*M$. The linear functions $\dot x^b$ correspond to differential forms $\xd x^b$ and the linear functions $p_b$ correspond to vector fields $\pa_{x^b}$. In other words, sections of $\sT M$ are linear functions on $\sT^*M$ and \emph{vice versa}. The degree of $p_a$ is $-\deg(x^a)$, and the degree of $\dot x^a$ is $\deg(x^a)$.
\begin{remark}
Note, however, that, because vector fields and one-forms have two natural $\cA(M)$-module structures, the left and the right, there are \emph{a priori} two possibilities of representing vector fields and one-forms as `submanifolds' in $\sT M$ and $\sT^*M$, respectively. The first is that $\dot x^a$ represent the left coefficients of vector fields and $p_a$ represent the right coefficients, so that $\big(x^a,f^b(x)\big)$ in $\sT M$ represents the vector field $X=\sum_af^a(x)\cdot\pa_{x^a}$, and $\big(x^a,g_a(x)\big)$ in $\sT^*M$ represents the one-form $\za=\sum_a\xd x^a\cdot g_a(x)$. Hence, for the pairing
$$\la\cdot,\cdot\ran_*:\sT M\ti\sT^*M\to\R\,,$$
we have
$$\la X,\za\ran_*=\sum_af^a(x)\,g_a(x).$$
The other possibility is that $\dot x^a$ represent the right coefficients of vector fields and $p_a$ represent the left coefficients of one-forms, so that $\big(x^a,f^b(x)\big)$ in $\sT M$ represents the vector field $X=\sum_a\pa_{x^a}\cdot f^a(x)$, and $\big(x^a,g_a(x)\big)$ in $\sT^*M$ represents the one-form $\za=\sum_a g_a(x)\cdot\xd x^a$. Hence, for the pairing $$\la\cdot,\cdot\ran^*:\sT^* M\ti\sT M\to\R\,,$$
we have
$$\la X,\za\ran^*=\sum_ag_a(x)\,f^a(x).$$
We will use the first convention.
\end{remark}
\no An attempt to define the wedge products in the graded superalgebras
$$\zW^\bullet(M)=\bigoplus_{k=0}^\infty\zW^k(M)\quad\text{and}\quad \fX^\bullet(M)=\bigoplus_{k=0}^\infty\fX^k(M)$$
could be, like in the purely even case, by viewing  differential forms and multivector fields on $M$ as superfunctions on $\zP\sT M$ and $\zP\sT^*M$, respectively.
The explicit sign rules would be
\[
x^a\cdot\pa_{x^b}=(-1)^{\zs_a\cdot(\zs_b+1)}\,\pa_{x^b}\cdot x^a\,,\quad \pa_{x^a}\we\pa_{x^b}=(-1)^{(\zs_a+1)(\zs_b+1)}\,\pa_{x^b}\we\pa_{x^a}
\]
and
\[x^a\cdot\xd {x^b}=(-1)^{\zs_a\cdot(\zs_b+1)}\,\xd {x^b}\cdot x^a\,,\quad \xd x^a\we\xd x^b=(-1)^{(\zs_a+1)(\zs_b+1)}\,\xd x^b\we\xd x^a\,.
\]
The above sign rules are called the \emph{Bernstein's sign convention}.

\mn In this paper, we will use the \emph{Deligne's sign convention}:
\be\label{sc1}
x^a\cdot\pa_{x^b}=(-1)^{\zs_a\cdot\zs_b}\,\pa_{x^b}\cdot x^a\,,\quad \pa_{x^a}\we\pa_{x^b}=-(-1)^{\zs_a\cdot\zs_b}\,\pa_{x^b}\we\pa_{x^a}
\ee
and
\be\label{sc2} x^a\cdot\xd {x^b}=(-1)^{\zs_a\cdot\zs_b}\,\xd {x^b}\cdot x^a\,,\quad \xd x^a\we\xd x^b=-(-1)^{\zs_a\cdot\zs_b}\,\xd x^b\we\xd x^a\,,
\ee
which, in our opinion, has many advantages over Bernstein's one (see, e.g., \cite[Appendix to \S 1]{Deligne:1999}). We will often skip the symbol "$\we$" and write simply $\xd x^a\,\xd x^b$.

\begin{remark} Note that the graded superalgebras $\zW^\bullet(M)$ (as well as $\fX^\bullet(M)$) have an additional $\N$-gradation, so the degrees of homogeneous elements have three components: $(\zs,w,r)\in\Z_2\ti\R\ti\N$, where $r\in\N$ is the \emph{rank} of the differential form, $r$-forms have rank $r$, $\zs$ is its parity, and $w$ is its weight. Deligne's convention comes from the $\Z_2^2$-gradation induced from the $\Z_2\ti\N$ gradation, in which $x^a$ has bi-degree $(\zs_a,0)$, $\xd x^a$ (thus $\dot x^a$) has bi-degree $(\zs_a,1)$, and the sign rule comes from the `scalar product' of the bi-degrees. Consequently, the product of homogeneous forms $\za$ and $\zb$ of degrees $(\zs_a,w_a,r_a)$ and $(\zs_b,w_b,r_b)$, respectively, has the property:
\be\label{SR}
\za\we\zb=(-1)^{\zs_a\zs_b+r_ar_b}\zb\we\za.
\ee
We can view this product as a product of superfunctions not on a supermanifold but rather on a $\Z_2^2$-manifold in the sense of \cite{Covolo:2016a}. The main difference with standard supermanifolds is that the parity does not determine the sign rules, so we must work with formal even variables that are not nilpotent, thus with formal power series. On the other hand, many important properties remain the same. For instance, an analog of the Bachelor-Gaw\c edzki theorem \cite{Gawedzki:1977} is still valid \cite{Covolo:2016b}. Note that the tangent and cotangent bundles of supermanifolds should be canonically considered rather as $\Z_2\ti\N$-graded, where the $\N$-gradation corresponds to the vector bundle structure.
\end{remark}

\mn The standard operations of the Cartan calculus on $\zW^\bullet(M)$ will be interpreted as the following homogeneous vector fields, viewed as graded derivations with respect to the $\Z_2\ti\R\ti\N$:-grading
\begin{enumerate}
\itemsep1em
\item The \emph{de Rham derivative}: $\rmd := \dot x^a \frac{\partial}{\partial x^a}$,  which has degree $(0,0,1)$;
\item The \emph{interior product}: If $X = X^a(x) \frac{\partial}{\partial x^a}$, then
$$i_X :=  X^a(x)\,{\partial_{\dot x^a}}\,,$$ which has degree $(\deg(X),-1)$,
\item The \emph{Lie derivative}:  If $X = X^a(x) \frac{\partial}{\partial x^a}$, then
\be\label{Ll}\Ll_X := [\rmd, i_X] = \xd\circ i_X+i_X\circ\xd= \dot x^b\, \frac{\partial X^a}{\partial x^b}(x)\, {\partial_{ \dot x^a}}+X^a(x) \pa_{x^a}\,,\ee
which has degree $(\deg(X),0)$.
\end{enumerate}
\no Note that these operations can be applied to all, not just homogeneous differential forms.
In particular, for a 1-form $\za=\sum_b\xd x^b\cdot g_b(x)$ on $M$ and a  vector field $X=\sum_af^a(x)\cdot\pa_{x^a}$, the insertion operator (the contraction) reads
$$i_X\za=\sum_af^a(x)g_a(x)\,.$$
If $\za$ and $X$ are homogeneous, then $i_X\za$ is a homogeneous function on $M$ of degree $\deg(X)+\deg(\za)$.
Contractions gives rise to linear functions on $\sT^*M$: identifying $\pa_{x^a}$ with the fiber coordinate $p_a$, we have
$$\zi_X(x^a,p_b)=\sum_af^a(x)\,p_a\,.$$
Similarly, for $\zb=\sum_b g_b(x)\cdot\xd x^b$ and $Y=\sum_a\pa_{x^a}\cdot f^a(x)$, we have
$$i_\zb Y=\sum_ag_a(x)f^a(x)\quad\text{and}\quad \zi_\zb(x^a,\dot x^b)=\sum_ag_a(x)\dot x^a\,.$$
For $i_X$ we will use the convention of the first type, while for $i_\za$, the convention of the second type.
\begin{proposition}
If $X,Y$ are homogeneous vector fields on a supermanifold $M$, then the graded commutators of the vector fields on $\sT M$ of the type in question are the following:
\beas
&2\,\rmd^2=[\rmd,\rmd] =0\,, \quad [i_X,i_Y] =0\,,\\
&[\xd,\Ll_X]=0\,, \quad [\Ll_X,i_Y]=i_{[X,Y]}\,,\quad [\Ll_X,\Ll_Y]=\Ll_{[X,Y]}\,.
\eeas
Moreover, the de Rham derivative satisfies
$$\xd(\za\we\zb)=\xd\,\za\we\zb+(-1)^k\za\we\xd\,\zb$$
for any $\zb$ and any $k$-form $\za$.
\end{proposition}
\no For two-forms and the homogeneous vector field $X=\sum_if^i(x)\cdot\pa_{x^i}$, we have
\beas &i_{X}\left(\sum_{i,j}\xd x^i\we\xd x^j\cdot\zw_{ij}(x)\right)=\sum_{a,i,j}f^a(x)\pa_{\dot x^a}
\big(\dot x^i\,\dot x^j\cdot\zw_{ij}(x)\big)\\
&=\sum_{i,j}f^i(x)\,\dot x^j\cdot \zw_{ij}(x)=\sum_{i,j}f^i(x)\,\xd x^j\cdot \zw_{ij}(x)\,.
\eeas
Here, due to skew-symmetry,
$$\zw_{ij}=-(-1)^{\zs_i\cdot\zs_j}\zw_{ji}\,.$$
We should also stress that the de Rham derivative we use looks slightly different from the standard notation in the even differential geometry, i.e., $$\xd f(x)=\rmd x^i \frac{\partial f}{\partial x^i}(x)\,.$$
The order of factors is crucial for supermanifolds.

\mn Note that the identity $[\Ll_X,i_Y]=i_{[X,Y]}$ can be seen as a definition of the Lie bracket of homogeneous vector fields. If $\zs_X$ is the parity of $X$ and $\zs_Y$ is the parity of $Y$, then in local homogeneous coordinates $(x^i)$ we have
\be\label{bracket}
\Big[\sum_iX^i\pa_{x^i},\sum_jY^j\pa_{x_j}\Big]=\sum_j\Big(X^i\pa_{x^i}\big(Y^j\big)-
\big(-1\big)^{\zs_X\cdot\zs_Y}Y^i\pa_{x^i}\big(X^j\big)\Big)\pa_{x^j}.
\ee

\section{Homogeneity Lie supergroups}\label{hls}
One can consider homogeneity structures on supermanifolds with additional structures such that these structures are compatible with the homogeneity structure. What compatibility means should be decided in each particular case.
For the group structure, we define \emph{homogeneity Lie groups} as follows.
\begin{definition}
A \emph{homogeneity Lie supergroup} is a group object in the category of homogeneity supermanifolds $\catname{HSMan}$.
In other words, a homogeneity Lie supergroup is a homogeneity supermanifold $G$ endowed with
\begin{enumerate}
\item a morphism of homogeneity manifolds
\be\label{zm}\zm:G\ti G\to G\ee
(the \emph{group multiplication}) satisfying the associativity property,
$$\zm\circ\big(\pr_1\ti\zm\circ(\pr_2\ti\pr_3)\big)=\zm\circ\big(\zm\circ(\pr_1\ti\pr_2)\ti\pr_3\big),$$
where both sides are morphisms $G\ti G\ti G\to G$;
\item a distinguished element $e\in|G|$ (the \emph{unit}), thus a distinguished morphism $\pr_e:G\to G$, being the composition $G\to\{e\}\to G$, and satisfying
satisfying the identity property,
$$\zm\circ\big(\id_G\ti\pr_e\big)=\id_G=\zm\circ\big(\pr_e\ti\id_G\big);$$
\item
a homogeneity diffeomorphism
$$\on{inv}:G\to G$$
(the \emph{inverse map}), satisfying the inverse property
\be\label{inv}\zm\circ\big(\id_G\ti\inv\big)=\pr_e=\zm\circ\big(\inv\ti\id_G\big).\ee
\end{enumerate}
\end{definition}
\no It is easy to see that the body $|G|$ of a homogeneity Lie supergroup $G$ is a homogeneity Lie group.
\begin{example}
Let us consider the Lie supergroup $\GL(1,1)$. As a supermanifold, it is an open submanifold of $\R^{2|2}$. In canonical coordinates $(x^1,x^2,\zx^1,\zx^2)$, this Lie supergroup is defined by  $x^1\cdot x^2\ne 0$. The coordinates in $\GL(1,1)$ are usually written in the form of an $2\ti 2$-matrix,
\be\label{gl} \begin{bmatrix}
x^1 & \zx^1\\
\zx^2& x^2
\end{bmatrix}\,,
\ee
and the group operations are formally the same as for the group $\GL(2;\R)$, with $e$ being the identity matrix.

\mn Now, we can make $G=\GL(1,1)$ into a homogeneity supermanifold by choosing the weight vector field
$$\nabla_G=a\zx^1\pa_{\zx^1}-a\zx^2\pa_{\zx^2}\,.$$
It is sometimes convenient to indicate the degrees of homogeneous coordinates in $\Z_2\ti\Z$ also in the form of a matrix. In the above case of homogeneity group $\GL(1,1)$, its homogeneity structure reads
$$ \deg\begin{bmatrix}
x^1 & \zx^1\\
\zx^2& x^2
\end{bmatrix}
=
\begin{bmatrix}
(0,0) & (1,a)\\
(1,-a) & (0,0)
\end{bmatrix}\,.
$$
The group multiplication is the matrix multiplication,
$$\begin{bmatrix}
x^1 & \zx^1\\
\zx^2& x^2
\end{bmatrix}
\cdot\begin{bmatrix}
y^1 & {\zh^1}\\
{\zh^2}& {y^2}
\end{bmatrix}
=
\begin{bmatrix}
x^1\cdot{y^1}+\zx^1\cdot{\zh^2} & x^1\cdot{\zh^1}+\zx^1\cdot{y^2}\\
\zx^2\cdot{y^1}+x^2\cdot{\zh^2}& \zx^1\cdot{\zh^1}+x^2\cdot{y^2}
\end{bmatrix}\,.
$$
This is a morphism of homogeneity supermanifolds, since
$$x^1\cdot{y^1}+\zx^1\cdot{\zh^2}\quad\text{and}\quad \zx^1\cdot{\zh^1}+x^2\cdot{y^2}$$
are of degree $(0,0)$,
$$x^1\cdot{\zh^1}+\zx^1\cdot{y^1}$$
is of degree $(1,a)$, and
$$\zx^2\cdot{y^1}+x^2\cdot{\zh^2}$$
is of degree $(1,-a)$. The inverse matrix of (\ref{gl}) is
$$ \begin{bmatrix}
\Big(x^1-\zx^1\,(x^2)^{-1}\,\zx^2\Big)^{-1} & -(x^1)^{-1}\,\zx^1\Big(x^2-\zx^2\,(x^1)^{-1}\,\zx^1\Big)^{-1}\\
-(x^2)^{-1}\,\zx^2\Big(x^1-\zx^1\,(x^2)^{-1}\,\zx^2\Big)^{-1}& \Big(x^2-\zx^2\,(x^1)^{-1}\,\zx^1\Big)^{-1}
\end{bmatrix}\,.
$$
All inverses above make sense, as $x^1$ and $x^2$ are invertible. It is also clear that the elements on the diagonal are of degree $(0,0)$,
$$-(x^1)^{-1}\,\zx^1\Big(x^2-\zx^2\,(x^1)^{-1}\,\zx^1\Big)^{-1}$$
is of degree $(1,a)$, and
$$-(x^2)^{-1}\,\zx^2\Big(x^1-\zx^1\,(x^2)^{-1}\,\zx^2\Big)^{-1}$$
is of degree $(1,-a)$. Hence, the inverse map is an automorphism of $G$ in the category $\catname{HSMan}$.
And finally, for the identity map, we have
$$ \pr_e\left(\begin{bmatrix}
x^1 & \zx^1\\
\zx^2& x^2
\end{bmatrix}\right)=\begin{bmatrix}
1 & 0\\
0& 1
\end{bmatrix}\,,
$$
so that $\pr_e$ is also a morphism in the category $\catname{HSMan}$.
\end{example}
\begin{example} Let us consider the Lie supergroup $\SL(2,1)$, understood as a submanifold in
$\R^{5|4}$ with coordinates written in the matrix form,
$$ X=\begin{bmatrix}
x_{11} & x_{12} & \zx_{13}\\
x_{21} & x_{22} & \zx_{23}\\
\zx_{31} & \zx_{32} & x_{33}
\end{bmatrix}\,,
$$
where $x$ are even and $\zx$ are odd, and defined by the equation $\ber(X)=1$. Here `$\ber$' denotes the Berezinian. We put a homogeneity structure on $\SL(2,1)$ by
$$\deg(X)=\begin{bmatrix}
(0,0) & (0,a) & (1,a+b)\\
(0,-a) & (0,0) & (1,b)\\
(1,-a-b) & (1,-b) & (0,0)
\end{bmatrix}\,.
$$
We will also write the matrix $X$ in the block form
$$\begin{bmatrix}
A & B\\
C & D
\end{bmatrix}\,,
$$
where
$$A=\begin{bmatrix}
x_{11} & x_{12}\\
x_{21} & x_{22}
\end{bmatrix}
\in\GL(2,\R)\,,
$$
etc.
First of all, $\SL(2,1)$ is a homogeneous submanifold of the manifold $\GL(2,1)$ with this homogeneity structure.
In other words, the condition $\ber(X)=1$ is compatible with the homogeneity structure.
In our notation, the Berezinian reads
$$\ber(X)=\det(A-BD^{-1}C)\det(D)^{-1}=\det(A)\det(D-CA^{-1}B)^{-1}=1\,.$$
Hence,
$$A-BD^{-1}C=\begin{bmatrix}
x_{11} & x_{12}\\
x_{21} & x_{22}
\end{bmatrix}
-x_{33}^{-1}
\begin{bmatrix}
\zx_{13}\zx_{31} & \zx_{13}\zx_{32}\\
\zx_{23}\zx_{31} & \zx_{23}\zx_{32}
\end{bmatrix}\,
$$
is of degree
\be\label{deg}\deg(A-BD^{-1}C)=\begin{bmatrix}
(0,0) & (0,a)\\
(0,-a) & (0,0)
\end{bmatrix}\,,
\ee
and its determinant, thus also $\ber(X)$, is of degree $(0,0)$.

We will show that the Lie supergroup $G=\SL(2,1)$, with the indicated homogeneity structure, is a homogeneity Lie supergroup. Tedious but easy calculations show that the multiplication $\zm:G\ti G\to G$ in $G$, i.e.,
the matrix multiplication
$$\begin{bmatrix}
A & B\\
C & D
\end{bmatrix}
\cdot\begin{bmatrix}
A' & B'\\
C' & D'
\end{bmatrix}
=
\begin{bmatrix}
AA'+BC' & AB'+BD'\\
CA'+DC' & CB'+DD'
\end{bmatrix}\,,
$$
is a morphism in the category of homogeneity manifolds. For instance, the degree of the matrix
$$
AA'+BC'=
\begin{bmatrix}
x_{11}x'_{11}+x_{12}x'_{21}+\zx_{13}\zx'_{31} & x_{11}x'_{12}+x_{12}x'_{22}+\zx_{13}\zx'_{32}\\
x_{21}x'_{11}+x_{22}x'_{21}+\zx_{23}\zx'_{31} & x_{21}x'_{12}+x_{22}x'_{22}+\zx_{23}\zx'_{32}
\end{bmatrix}
$$
is
\be\label{degree}\begin{bmatrix}
(0,0) & (0,a) \\
(0,-a) & (0,0)
\end{bmatrix}\,,
\ee
and the degree of the matrix
$$AB'+BD'=\begin{bmatrix}
x_{11}\zx'_{13}+x_{12}\zx'_{23} +\zx_{13}x'_{33}\\
x_{21}\zx'_{13}+x_{22}\zx'_{23} +\zx_{23}x'_{33}
\end{bmatrix}\,,
$$
is
$$\begin{bmatrix}
(1,a+b)  \\
(1,b)
\end{bmatrix}\,,
$$
as it should be.

\mn For the inversion, we have
$$X^{-1}=\begin{bmatrix}
(A-BD^{-1}C)^{-1} & -A^{-1}B(D-CA^{-1}B)^{-1}\\
-D^{-1}C(A-BD^{-1}C)^{-1} & (D-CA^{-1}B)^{-1}
\end{bmatrix}\,.
$$
In view of (\ref{deg}), similarly to the previous example, we get that the degree of $(A-BD^{-1}C)^{-1}$ is (\ref{degree}).
As the degree of $D^{-1}C$ is $[(1,-a-b),(1,-b)]$, it is easy to see that the degree of $D^{-1}C(A-BD^{-1}C)^{-1}$ also $[(1,-a-b),(1,-b)]$. Similarly, we get the correct results for the second column.
\end{example}
\no The fact that the group multiplication (\ref{zm}) is a morphism of homogeneity manifolds means exactly that the weight vector field $\n$ is \emph{multiplicative}; in particular, $\n$ vanishes at $e$. Then, the inverse and the identity are automatically morphisms of homogeneity manifolds. For instance, the multiplicativity of $\n$ is equivalent to
\be\label{weigtcomm}\zm\circ\big(\Exp(s\n)\ti\Exp(s\n)\big)=\Exp(s\n)\circ\zm,\ee
and we get from (\ref{inv},
$$\Exp(s\n)\circ\inv=\inv\circ\Exp(s\n),$$
so  $\inv$ respects $\n$. The tangent lift $\dt\n$ is a linear vector field on $\sT G$, which is vertical at $e$, that defines a graded vector structure on the Lie superalgebra $\g=\sT_eG$. From (\ref{weigtcomm}) it is easy to derive the fact that this makes $\g$ into a graded Lie superalgebra.
We can sum up these observations as follows.
\begin{proposition}
A homogeneity Lie supergroup is a Lie supergroup $G$ equipped with a multiplicative weight vector field. It induces a graded Lie superalgebra structure on the Lie algebra $\g$ of $G$.
\end{proposition}
\begin{remark}
We will not discuss the homogeneity principal superbundles here, but let us make a remark. As local trivializations of the $\R^\ti$-principal bundle $\zt:\R^{n|m}(\zm)^\ti\to\mathbb{P}\R^{n|m}(\zm)$ (cf. Example \ref{hps}), we can take
$\psi_i:\zt^{-1}(U_i)\to U_i\ti\R(w_i)^\ti$, which in coordinates $(y^j,\zx^a)$ on $\R^{n|m}(\zm)^\ti$ and $(z^k,\zvy^a)$ on $U_i$ read
$$\psi_i(y^j,\zx^a)=\left(\frac{y^1}{y^i},\dots,\dots,\frac{y^{i-1}}{y^i},\frac{y^{i+1}}{y^i},
\dots,\frac{y^n}{y^i},\frac{\zx^1}{y^i},\dots,\frac{\zx^m}{y^i},y^i\right)\,.$$
The typical fiber is therefore the homogeneity supermanifold $\R[w_i]^\ti$, the even linear homogeneity supermanifold $\R[w_i]$ with the zero-section removed. The corresponding weight vector field  is $\n=w_is\pa_s$, where $s$ is the standard coordinate on $\R$ (thus $\R^\ti$). But $F=\R(w)^\ti$ is not a homogeneity Lie group for $w\ne 0$, since the group multiplication $m:F\ti F\to F$ is not a morphism of homogeneity manifolds, as the degree of $s_1s_2$ in $F\ti F$ is $2w$, not $w$.

This suggests that the definition of a homogeneity $G$-principal bundle $\zt:P\to M$ should take this into account and cannot be a carbon copy of the standard definition. Consequently, the local trivializations should not be of the form $U\ti G$, but rather of the form $U\ti F$, where $U\ti F$ is the Cartesian product of homogeneity supermanifolds, with a free and transitive homogeneity action $A:G\ti F\to F$. In our case $G=\R^\ti$ and $F=\R(w_i)^\ti$, with the obvious action $A(t)(s)=ts$. Note here another, completely different approach to graded Lie groups \cite{Benoit:2022}.
\end{remark}
\section{Homogeneous Poincar\'e Lemma}
We know from Lemma \ref{ul} that, in homogeneous coordinates $(x^i)$ with degrees $(\zl_i)$, the partial derivative $\pa_{x^k}(f)$ is homogeneous of degree $(\zl-\zl_k)$ if $f$ is homogeneous of degree $\zl$. A converse of the latter is the following.
\begin{theorem}\label{primitive} Let $\nabla_M$ be a weight vector field which, in local coordinates $(x^i)$ in a neighbourhood of $m\in|M|$, reads $\nabla_M=\sum_{i=1}^nw_i\cdot x^i\,\pa_{x^i}$, and let $g(x)$ be a homogeneous function of degree $\zl=(\zs,w)$ . Then, there is a function $f$ of degree $\zl+\zl_k$, $f(m)=0$, such that
    $$\frac{\pa f}{\pa x^k}(x)=g(x),$$
\begin{description}
\item{(a)} \ for $x^k$ odd, if and only if $\pa_{x^k}(g)=0$;
\item{(b)} \ for $x^k$ even, if $(w_k\cdot x^k)(m)=0$ (in particular, if $\nabla_M(m)=0$).
\end{description}
\end{theorem}
\begin{proof}
It is clear from Lemma \ref{ul} that if $f$ is homogeneous and $\pa_{x^k}(f)=g$, then $f$ is of weight $w+w_k$. Permuting coordinates, we can assume that $k=1$.

\mn If $x^1=\zx$ is odd, then, the partial derivative $\frac{\pa f}{\pa\zx}=g$ does not depend on $\zx$ any longer, so $\pa_{\zx}(g)=0$ is a necessary condition. But in this case, it is enough to put $f=\zx\cdot g$, which is, clearly, a homogeneous function of weight $(w+w_1)$ and $f(m)=0$.

\mn In the case of $x^1$ even, define $f$ by
$$f(x)=\int_{x^1(m)}^{x^1}g(s,x^2,\dots,x^n)\xd s\,.$$
It is obvious that $\frac{\pa f}{\pa x^1}(x)=g(x)$ and $f$ have the parity of $g$, so it remains to check the weight of $f$. For,
\beas\nabla_M(f)(x)&=&w_1\cdot x^1\,\frac{\pa f}{\pa x^1}(x)+\sum_{a=2}^nw_a\cdot x^a\,\frac{\pa f}{\pa x^a}(x)\\
&=&w_1\cdot x^1\,g(x)+\int_{x^1(m)}^{x^1}\left(\sum_{a=2}^n\left(w_a\cdot x^a\,\frac{\pa g}{\pa x^a}\right)(s,x^2,\dots,x^n)\right)\xd s\\
&=&w_1\cdot x^1\,g(x)+\int_{x^1(m)}^{x^1}\left(w\cdot g(s,x^2,\dots,x^n)-w_1\cdot s\cdot\frac{\pa g}{\pa x^1}(s,x^2\cdot,x^n)\right)\xd s\\
&=&w_1\cdot x^1\,g(x)+w\cdot f(x)-w_1\cdot\int_{x^1(m)}^{x^1} s\cdot\frac{\pa g}{\pa x^1}(s,x^2\dots,x^n)\,\xd s\,.
\eeas
Integrating now by parts, we get
\beas \nabla_M(f)(x)&=&w_1\cdot x^1\,g(x)+w\cdot f(x)\\
&&-w_1\cdot\left((s\cdot g(s,x^2,\dots,x^n))\,\big|^{x^1}_{x^1(m)}-\int_{x^1(m)}^{x^1}g(s,x^2\dots,x^n)\,\xd s\right)\\
&=&w_1\cdot x^1\,g(x)+w\cdot f(x)-w_1\cdot x^1\,g(x)\\
&&+w_1\cdot x^1(m)\,g(x^1(m),x^2,\dots,x^n)+w_1\cdot f(x)\\
&=&(w+w_1)\cdot f(x)+w_1\cdot x^1(m)\,g(x^1(m),x^2,\dots,x^n)=(w+w_1)\cdot f(x)\,.
\eeas
\end{proof}
\begin{remark}
There is no extension of the above theorem to the case $w_k\cdot x^k(m)\ne 0$. Consider $M=\R$ with the weight vector field $\nabla_M=x\,\pa_x$, so that $x$ is of weight $1$. In a neighbourhood of $m=1$, the function $g(x)=1/x$ is homogeneous of weight $-1$, but there is no smooth homogeneous function $f$ in a neighbourhood of $m=1$ such that $f'(x)=g(x)$. Indeed, any smooth function $f$ satisfying $f'(x)=g(x)$ in a neighbourhood of $1$ is of the form $f(x)=\ln(x)+c$, where $c$ is a constant. Neither of them is homogeneous.
\end{remark}
\no Later on, we will use a modified version of Moser's trick. To this end, we will need the following homogeneous Poincar\'e Lemma.
\begin{lemma}\label{Mt}\textbf{(Homogeneous Poincar\'e Lemma)}
Let $(M,\nabla_M)$ be a homogeneity supermanifold of total dimension $k$, let $m\in|M|$, and let
$\nabla_M=\sum_{i=1}^kw_i\cdot x^i\,\pa_{x^i}$ in a neighbourhood of $m$. In this neighbourhood, take a closed homogeneous $n$-form $\zw$, $n>0$, of degree $\zl=(\zs,w)$.
\begin{enumerate}
\item
If $\nabla_M(m)=0$, then we can find a homogeneous $(n-1)$-form $\za$ of degree $\zl$ such that $\zw=\xd\za$. Moreover, we can choose $\za$ such that $\za(m)=0$.
\item
If $\nabla_M(m)\ne 0$, then there is such a form $\za$ for $n>1$. For $n=1$, we can find a homogeneous function $f$ such that $\zw=\xd f$, except for the case $w=0$. However, the function $f$ does not vanish at $m$ (except $\zw=0$, of course).
\item If $\nabla_M(m)\ne 0$ and a homogeneous 1-form $\zw$ has weight 0, then in local coordinates $(x,y^i)$ such that $x(m)=1$ and $y^i(m)=0$ (see Corollary \ref{cor}),
$$\zw(x,y)=c\,\xd x/x+\xd g(y),$$
for a constant $c\in\R$, and $g=g(y)$ being a function depending on coordinates $(y^i)$ only. In other words,
$$\zw=\xd f,\quad f(x,y)=c\,\ln(x)+g(y)+c_1.$$
The function $f$ is not homogeneous if $c\ne 0$.
\end{enumerate}
\end{lemma}
\begin{proof}
\textbf{1.} Assume first that $\nabla_M(m)=0$. If $w\ne 0$, then we can simply put
$$\za=\frac{1}{w}\cdot i_{\nabla_M}\zw\,.$$
As $\nabla_M$ is of degree 0 and $\zw$ is of degree $\zl$, the $(n-1)$-form $\za$ is of degree $\zl$ and vanish at $m$. Moreover,
$$\xd \za=\frac{1}{w}\cdot\xd\, i_{\nabla_M}\zw=\frac{1}{w}\cdot\Ll_{\nabla_M}\zw=\zw\,.$$

\medskip\noindent
It remains to consider the case $w=0$, so let us look for a $(n-1)$-form $\za$ satisfying
\be\label{D}\deg(\za)=\zl\,,\ \xd\za=\zw\,,\ \za(m)=0\,.\ee
We will use the induction with respect to the total dimension $k$ of $M$.

\medskip\noindent
The first case we consider is $M$ purely odd.
If $k=1$, then there is a single odd coordinate $\zx$ around $m\in|M|=\{ m\}$. If $\zw$ is a closed $n$-form, then $\zw=(a\cdot\zx+b)(\xd\zx)^n$, $a,b\in\R$. But $\xd\zw=a\cdot(\xd\zx)^{n+1}=0$, so $a=0$ and $\zw=b\cdot(\xd\zx)^n$. Hence, $\zx$ is of weight $0$, and we can put $\za=b\cdot\zx\,(\xd\zx)^{n-1}$.

\medskip\noindent
Assuming that we have proved the Theorem for all purely odd homogeneity supermanifolds $M$ of dimension $k$, let us take a purely odd homogeneity supermanifold $M$ of dimension $(k+1)$ and a closed homogeneous $n$-form $\zw$ on $M$ of weight $0$. For $\zx=x^1$ we can write $\zw$ uniquely in the form
$$\zw=\sum_{j=0}^n(\xd\zx)^j\we\zw_j+\sum_{j=0}^n\zx\,(\xd\zx)^{j}\we\zb_j\,,$$
where the forms $\zw_j$ and $\zb_j$ depend only on coordinates $x^2,\dots,x^k$. Of course, $\zw_j$ is a homogeneous $(n-j)$-form of weight $-jw_1$. Since
\beas\xd\zw&=&\sum_{j=0}^n(-1)^j(\xd\zx)^j\we\xd\zw_j+\sum_{j=0}^n(\xd\zx)^{j+1}\we\zb_j+
\sum_{j=0}^n(-1)^{j}\zx\,(\xd\zx)^{j}\we\xd\zb_j\\
&=&\xd\zw_0+\sum_{j=1}^n(\xd\zx)^j\we\big(\zb_{j-1}+(-1)^j\xd\zw_j\big)+(\xd\zx)^{n+1}\we\zb_n+
\sum_{j=0}^n(-1)^{j}\zx\,(\xd\zx)^{j}\we\xd\zb_j=0\,,\eeas
we get $\xd\zw_0=0$ and $\zb_j=(-1)^{j}\,\xd\zw_{j+1}$ for $j=0,\dots,n-1$, and $\zb_n=0$. Since $\zw_0$ depends only on coordinates $x^i$ with $i>1$, we can view $\zw_0$ as an $n$-form on the homogeneity submanifold $x^1=0$ in $M$ with the weight vector field $\nabla_0=\sum_{i=2}^kw_i\cdot x^i\pa_{x^i}$ vanishing at $0$ (all coordinates are odd).

Now, applying the inductive assumption, we can find a homogeneous $(n-1)$-form $\za_0$ of weight $0$ and vanishing at $0$ such that $\xd\za_0=\zw_0$.
Viewing naturally $\za_0$ as an $(n-1)$-form on $M$, vanishing at $m$, we can take
$$\za=\za_0+\sum_{j=1}^n\zx\,(\xd\zx)^{j-1}\we\zw_{j}\,.$$
Clearly, $\za$ is homogeneous of weight $0$ and  vanishes at $m$. Moreover,
$$\xd\za=\xd\za_0+\sum_{j=1}^n(\xd\zx)^{j}\we\zw_j+\sum_{j=1}^n(-1)^{j-1}\zx\,(\xd\zx)^{j-1}\we\xd\zw_j
=\zw_0+\sum_{j=1}^n(\xd\zx)^{j}\we\zw_j+\sum_{j=0}^{n-1}\zx\,(\xd\zx)^{j}\we\zb_{j}=\zw\,.$$
This ends the inductive proof of the homogeneous Poincar\'e Lemma for purely odd supermanifolds.

\medskip\noindent
Now, we pass to the case when there exist even coordinates in a neighbourhood of $m$ and $\nabla_M(m)=0$.
Assume that $x=x^1$ is even, and denote the rest of the homogeneous coordinates around $m$ by $y^i$, $i=2,\dots,k$,
so that $\nabla_M=w_1\cdot x\,\pa_x+\nabla_1$,
$$\nabla_1=\sum_{i=2}^{k}w_i\cdot y^i\,\pa_{y^i}\,.$$
As $\nabla_M(m)=0$, we can take the coordinates $(x,y)$ vanishing at $0$.

For $k=1$ there are no coordinates $y^i$, and the $n$-form $\zw$ must be, actually, a one-form. We have $\zw(x)=g(x)\,\xd x$  for a homogeneous function $g(x)$ of weight $-w_1$. The one-form $\zw$ is automatically closed for any $g(x)$. If $w_1=0$, then we can put $\za=f(x)$ for any primitive function $f$ of $g$, $f'=g$. We can always choose $f(m)=0$. Thus we can reduce ourselves to the case $w_1\ne 0$, so that $w_1\cdot x\,\pa_x(g)=-w_1\cdot g$.
Solutions of the above equation on $(0,+\infty)$ or $(-\infty,0)$
have the form $g(x)=c\cdot x^{-1}$, where $c$ is a constant, which are not smooth at $0$; a contradiction, so there are no 1-forms $\zw$ satisfying the requirements. This finishes the proof for $k=1$, and we can pass to the inductive step.

\mn Let us consider $M$ of total dimension $(k+1)$ and a homogeneous $n$-form $\zw$ on $M$ of weight 0. We can write $\zw$ as $\zw=\xd x\we\zn_1+\zn_2$, where
$\zn_1$ and $\zn_2$ are homogeneous $(n-1)$ and $n$-forms of degrees $-w_1$ and $0$, respectively, which do not contain $\xd x$. Put
$$\zn_1=\sum_\zm g_\zm(x,y)\,(\xd y)^\zm\,,$$
where $\zm=(\zm_2,\dots,\zm_{k+1})\in\N^{k}$ is a multi-index such that $\sum_i\zm_i=n-1$ and
$$(\xd y)^\zm=(\xd y^2)^{\zm_2}\we\dots\we(\xd y^{k+1})^{\zm_{k+1}}\,.$$ Of course, $\zm_i\le 1$ if $y^i$ is even. The function $g_\zm(x,y)$ is homogeneous of weight
$-w_1-|\zm|$, where $|\zm|=\sum_{i=2}^{k+1}w_i\cdot \zm_i$.

According to Theorem \ref{primitive}, we can find homogeneous functions $f_\zm(x,y)$ of weight $-|\zm|$ such that $\pa_x(f_\zm)=g_\zm$. Moreover, $f_\zm(0,y)=0$. Put
$$\zb=\sum_\zm f_\zm(x,y)\,(\xd y)^\zm\,.$$
The $(n-1)$-form $\zb$ is homogeneous of weight 0, vanishes at $m$, and $\zw_1=\zw-\xd\zb$ is a homogeneous $n$-form of weight 0.
A simple observation is that $\zw_1$ does not contain $\xd x$. But $\xd\zw_1=0$, so that the coefficients of $\zw_1$ do not depend on $x$. Consequently, $\zw_1$ is a homogeneous $n$-form of weight 0 in coordinates $(y^i)$ with the weight vector field $\nabla_1$. As $\nabla_1(m)=0$, we can apply the inductive assumption and find a homogeneous $(n-1)$-form
$\za_1$ of weight 0, vanishing at $m$ and depending on $y^i$ only, such that $\zw_1=\xd\za_1$. Now, it is enough to put $\za=\za_1+\zb$.

\medskip\noindent
\textbf{2.} Now, we will consider the case $\nabla_M(m)\ne 0$.
As we know from Corollary \ref{cor}, we can choose homogeneous coordinates $(x,y^i)$ in a neighbourhood of $m$ such that $\nabla_M=x\,\pa_x$, $x(m)=1$, $y^i(m)=0$.

\mn Take a homogeneous closed $n$-form $\zw$ of weight $w$ on $M$, so that
$$\zw=\xd x\we\zn_1+\zn_2\,,$$
where $\zn_1$ and $\zn_2$ do not contain $\xd x$. Since $\zn_1$ is homogeneous of weight $(w-1)$ and $\zn_2$ is homogeneous of weight $w$, it is easy to see that $\zn_1=x^{w-1}\zw_1$ and  $\zn_2=x^w\,\zw_2$, where $\zw_1$ and $\zw_2$ are forms in the coordinates $(y^i)$ only, thus homogeneous of weight 0. We have then
$$\zw=x^{w-1}\,\xd x\we\zw_1+x^w\,\zw_2\,,$$
and
$$\xd\zw=\xd\left((x^{w-1}\,\xd x)\we\zw_1+x^w\,\zw_2\right)=-x^{w-1}\,\xd x\we\xd\zw_1+w\, x^{w-1}\xd x\we\zw_2+x^w\,\xd\zw_2=0\,.$$
Hence, $\xd\zw_2=0$ and $w\,\zw_2=\xd\zw_1$.
As $\zw_2$ is an $n$-form in coordinates $(y^i)$ of weight 0, we can apply our homogeneous Poincar\'e Lemma, proved already for the case $\nabla_M(m)=0$, to $\zw_2$ and write $\zw_2=\xd\zb$, where $\zb$ is a homogeneous $(n-1)$-form
in coordinates $(y^i)$, vanishing at $m$. Hence, $\xd\zw_1=w\,\xd\zb$, so $\zw_1=w\,\zb+\xd\zg$, for a $(n-2)$-form $\zg$ in coordinates $(y^i)$ with $\zg(m)=0$.

\mn If $w\ne 0$, we can write
\beas\zw&=&x^{w-1}\,\xd x\we(w\,\zb+\xd\zg)+x^w\,\xd\zb=\xd(x^w)\we\zb+x^w\,\xd\zb+\xd(x^w/w)\we\xd\zg\\
&=&\xd(x^w\,\zb)-\xd(\xd(x^w/w)\we\zg)=\xd\left(x^w\,\zb-\xd( x^w/w)\we\zg\right)\,,\eeas
so we can put
$$\za=x^w\,\zb-\xd( x^w/w)\we\zg\,.$$
The $(n-1)$-form $\za$ is homogeneous of weight $w$ and vanishes at $m$, since $\zb$ and $\zg$ do.

\mn In the case $w=0$, we have $\zw_2=0$, $\xd\zw_1=0$. If $n>1$, then applying part 1. of our homogeneous Poincar\'e Lemma, we can write $\zw_1=\xd\za_1$ for a homogeneous $(n-2)$-form in coordinates $(y^i)$, vanishing at $m$. As $\zw=x^{-1}\,\xd x\we\xd\za_1$, we can put
$$\za=-\xd x\we(\za_1/x)\,.$$ It is homogeneous of weight 0 and vanishes at $m$.

\mn On the other hand, if $n=1$, then
$\zw=(\xd x/x)\,f(y)+\zn(y)$, where $\zn$ is a 1-form and $f$ is a function, both depending only on variables $(y^i)$.
As $\xd\zw=(\xd x/x)\we\xd f(y)+\xd\zn(y)=0$, we get that $f=c$ is a constant and $\xd\zn=0$. Applying part 1 of our homogeneous Poincar\'e Lemma again, we can write $\zn=\xd g$ for a function $g=g(y)$. But all functions $\za$ such that $\xd\za=\zw=c\,\xd x/x+\xd g$ are of the form $c\,\ln(x)+g+c_1$, where $c_1$ is another constant. Such a function is homogeneous if and only if $c=0$, and we end up with $\zw$ that depends only on coordinates $(y^i)$ of weight 0; the case which was already solved. This proves the last part of the theorem.

\end{proof}
\section{Homogeneous symplectic forms}
Let us recall that a 2-form $\zw$ on a homogeneity supermanifold $(M,\n)$ is {homogeneous} of degree $\zl=(\zs,w)\in\Z_2\ti\R$ if $\zw$ has parity $\zs$ and $\Ll_\n\zw=w\cdot\zw$.
\begin{example}
If $M$ is a homogeneity supermanifold of total dimension $n$, then on $\sT^*[\zl]M$ there is a canonical symplectic form $\zw_M[\zl]$ of degree $\zl$. In the adapted coordinates $(x^a,p_b)$ on $\sT^*[\zl]M$, induced by homogeneous coordinates $(x^a)$ on $M$, we have $\deg(p_a)=\zl-\deg(x^a)$ and
$$\zw_M[\zl]=\sum_a\xd x^a\we \xd p_a\,.$$
In particular, the canonical symplectic form on $\zP\sT^*M=\sT^*[(1,0)]M$ is odd. Moreover, $\zw_M[\zl]=-\xd\zvy_M[\zl]$, where $$\zvy_M[\zl]=\sum_a\xd x^a\cdot p_a$$
is the canonical Liouville 1-form of degree $\zl$ on $\sT^*[\zl]M$.
\end{example}
\no Any two-form $\zw$ on a homogeneity supermanifold $M$ which is homogeneous of degree $\zl=(\zs,w)$ induces a morphism of homogeneous vector superbundles
$$\zw^\sharp:\sT M\to\sT^*[\zl]M$$
by contraction, $X\mapsto i_X\zw$. In local adapted coordinates, if $\zw=\xd x^a\we\xd x^b\cdot \zw_{ab}(x)$, then
\be\label{e4}\zw^\sharp(x^a,\dot x^b)=\left(x^a, \sum_c\dot x^c\zw_{cb}(x)\right)\,.\ee
Indeed, we have
$$\deg(\zw_{ab})=\zl-\deg(x^a)-\deg(x^b),$$
so
\beas &\deg\big(\dot x^a\zw_{ab}(x)\big)=\deg(\zw_{ab})+\deg(x^a)\\
&=\zl-\deg(x^a)-\deg(x^b)+\deg(x^a)=\zl-\deg(x^b),
\eeas
which agrees with $\deg(p_b)$ in $\sT^*[\zl]M$.
Note also that the form $\zw$, for any $m\in|M|$, induces an $\R$-linear map of degree $\zl$, $$\zw(m)^\sharp:\sT_mM\to\sT^*_mM,$$
between homogeneity super-vector spaces. If $\zp_M$ is the canonical surjection $\zp_M:\cA(M)\to C^\infty(|M|)$,
then $\zw(m)^\sharp$ is represented by the matrix $[\zw_{ab}(m)]\in\on{gl}(n;\R)$, where $\zw_{ab}(m)=\zp_M(\zw_{ab})(m)$.
\begin{definition}
We say that a 2-form $\zw$ on a supermanifold $M$ of total dimension $n$ has rank $r$, if $\zw^\sharp$ maps $\sT M$ onto a rank $r$ vector subbundle in $\sT^*M$.
A closed 2-form $\zw$ on $M$ we call \emph{symplectic} if $\zw^\sharp$ is an isomorphism of vector superbundles, i.e., $\zw^\sharp$ is of rank $n$.
\end{definition}
\no It is easy to see that $\zw^\sharp$ is of rank $r$ if and only if the $n\ti n$-matrix $[\zw_{ab}(x)]$ with coefficients in local smooth function on $M$ is of rank $r$. The following well-known lemma reduces the concept of this rank to the case of real matrices.
\begin{lemma}\label{l1}
Let $\cM_{n\ti n}(\cA)$ be the space of $n\ti n$-matrices with coefficients in the superalgebra $\cA(\R^{p|q})$ of functions (superfields) on $M=\R^{p|q}$.
Then, a matrix $X=[a^{i}_{j}]\in\cM_{r\ti r}(\cA)$ has rank $r$ if and only if the matrix $[(a^i_{j})(m)]$ is of rank $r$ for every $m\in\R^p$. In particular, $X$ is invertible in $\cM_{n\ti n}(\cA)$ if $[(a^i_{j})(m)]\in\GL(n;\R)$ for all $m\in\R^p$.
\end{lemma}

\begin{corollary}
A closed $2$-form $\zw$ of degree $\zl=(\zs,w)$ on a $\zG$-homogeneity supermanifold $M$ has rank $r$ if and only if the maps $\zw(m)^\sharp:\sT_mM\to\sT^*[\zl]_mM$ are linear homogeneous maps of rank $r$ between graded vector spaces for all $m\in|M|$. Alternatively, the skew bilinear form $\zw(m):\sT_mM\ti\sT_mM\to\R$ is of rank $r$. Here, $\zw(m)\in\zL^2\,\sT_m^*M$ and
$$\zw(m)(\pa_{x^a},\pa_{x^b})=(i_{\pa_{x^b}}\circ i_{\pa_{x^a}})\zw(m)\,.$$
\end{corollary}
\no The skew-symmetry of $\zw(m)$ means that
$$\zw(m)(e,e')=-(-1)^{(\zs+\zs_0)\cdot(\zs+ \zs_1)}\,\zw(m)(e',e)$$
for homogeneous vectors $e,e'$ with $\deg(e)=(\zs_0,w_0)$ and $\deg(e')=(\zs_1,w_1)$.
\begin{proof}
Direct application of Lemma \ref{l1}.
\end{proof}
\begin{theorem}\label{t2}
Let $V$ be a $(-\zG)$-graded vector superspace of total dimension $n$, and let $g:V\ti V\to\R$ be a bilinear skew-symmetric two-form of degree $\zl=(\zs,w)\in\Z_2\ti\R$ and rank $r$.
Then, there is a homogeneous basis $(q^1,\dots,q^a,p_1,\dots,p_a,y_1,\dots,y_b,z^1,\dots,z^c)$ of $V$ such that $2a+b=r$ and
\be
g=\sum_{i=1}^{a} (p_i)^*\we (q^i)^*+\,\sum_{j=0}^{b}\ze^j(y_j^*\we y_j^*)\,,
\ee
where $(p_i)^*,(q^i)^*,(y^j)^*,(z^k)^*$ is the dual basis in the $\zG$-graded dual vector superspace $V^*$.
In particular,
$$\ze^j=\pm 1\,,\quad\deg(p_i)=\zl-\deg(q^i)\,,\quad 2\deg(y^j)=(0,w)\,.$$
The vectors $y^j$ must be odd; thus, they can appear only if $g$ is even.
\end{theorem}
\begin{proof}
This is a graded super-version of a well-known theorem in elementary linear algebra. We will use a modified Gram-Schmidt process to bring $g$ into the desired form by induction with respect to $n$.
For $n=1$ the theorem is trivial, so suppose the theorem is true for $\dim{V}<n$ and take $V$ such that $\dim(V)=n$.

\medskip\noindent
Let $V_0=\Ker(g^\sharp)$ be the kernel of $g^\sharp:V\to V^*$. Since $g$ is homogeneous, it is easy to see that $V_0$ is a graded vector subspace of $V$, so we can choose a complementary vector subspace $V'$ spanned by homogeneous vectors, say $e^1,\dots,e^k$. There are two possibilities:
\begin{enumerate}
\item $g(e^1,e^1)=0$\,,
\item $g(e^1,e^1)\ne 0$\,.
\end{enumerate}
\textbf{1.} In the first case we can find $e^i\in V'$, $i\ne 1$, such that $g(e^1,e^i)=c\ne 0$; otherwise $g=0$ and the theorem is trivial. We can permute the homogeneous base such that $e^i=e^2$. Moreover, multiplying $e^2$ by $1/c$, we can assume that $g(e^1,e^2)=1$. It is also clear that $\deg(e^1)=-\zl-\deg(e^2)$. Denote $e^1$ by $p_1$ and $e^2$ by $q^1$. Let $V_1$ be the `orthogonal complement' of $\on{span}\la p_1,q^1\ran$ in $V'$ with respect to the non-degenerate bilinear form $g\,\big|_{V'}$, i.e., $x\in V'$ is in $V_1$ if and only if $g(p_1,x)=0$ and $g(q^1,x)=0$. Since $g$ is homogeneous, $V_1$ is a homogeneous (graded) subspace of $V'$, and as a homogeneous basis of $V_1$ we can take
$$\big(f^j=e^j-g(e^1,e^j)e^2+(-1)^{(\zs+\zs_a)\cdot(\zs+\zs_b)}\,g(e^2,e^j)e^1\big)_{(k\ge j>2)}\,.$$
The vectors $f^j$ are indeed homogeneous, $\deg(f^j)=\deg(e^j)$, since $g(e^1,e^j)e^2$ and $g(e^2,e^j)e^1$ are of degree $\deg(e^j)$.
Suppose that $g(e^1,e^j)$ is non-zero. This is possible only if $g^{1j}\ne 0$, i.e., $$\zl+\deg(e^1)+\deg(e^j)=\zl-\zl-\deg(e^2)+\deg(e^j)=0\,.$$
Hence, $\deg\big(g(e^1,e^j)e^2\big)=\deg(e^2)=\deg(e^j)$. Similarly we prove $\deg\big(g(e^2,e^j)e^1\big)=\deg(e^j)$, thus  $f^j$ are homogeneous for $k\ge j>2$.

\mn In the basis dual to the homogeneous basis $p_1,q^1,f^3,\dots,f^n$ of $V$, the bilinear form $g$ reads
$$g=(p_1)^*\we(q^1)^*+g^1\,,$$
where $g^1$ is the restriction of $g$ to $V_1$, thus homogeneous with degree $\zl$ and of rank $(r-2)$. As $\dim(V_1)\le n-2$, we can now apply the inductive assumption.

\mn\textbf{2.} We can find $a\in\R$ such that $g(ae^1,ae^1)=\ze_1=\pm 1$.
Put $y^1=ae^1$, which is clearly odd. We have $-2\deg(e^1)=\zl$ that implies that the form $g$ is even and of weight $w=-2w_1$. Similarly to above, we prove that the orthogonal complement
$$V_1=\{x\in V\,|\, g(e^1,x)=0\}$$
of $\la e^1\ran$ in $V'$ is a graded vector subspace of $V'$ with a homogeneous basis, say, $f^2,\dots,f^k$. In the basis dual to $y^1,f^2,\dots,f^k$, the skew bilinear form $g$ reads
$$g=\ze^1(y_1^*\we y_1^*)+g^1\,,$$
where $g^1$ is a skew-symmetric non-degenerate bilinear form  on $V_1$ of degree $\zl$. As $\dim(V_1)\le n-1$, we finish the proof, applying the inductive assumption to $(V_1,g^1)$.

\end{proof}
\subsection{The homogeneous Darboux theorem}
Let $M$ be a $\zG$-homogeneity supermanifold and let $\zw$ be a symplectic form of degree $\zl=(\zs,w)$ on $M$. By \emph{Darboux coordinates} for $(M,\zw)$ we will understand local homogeneous coordinates $(p_1,\dots,p_j,q^1,\dots, q^{j},y_1,\dots y_k)$ on $M$, with weights in $\zG$, such that
\be\label{Darboux}
\zw=\sum_{i=1}^j  \rmd p_i\we \rmd q^i  +\sum_{l=1}^k\ze^l\,(\rmd y_l\we \xd y_l)\,,
\ee
where $\ze^l=\pm 1$.
\begin{remark}
This should be compared with the Darboux theorem in non-graded symplectic supergeometry, see for example Schwarz \cite{Schwarz:1996}, proofs of which are sketched by Kostant \cite{Kostant:1977} and Shander \cite{Shander:1983}. We remark also that Khudaverdian gave a `simple' proof of the Darboux theorem for odd symplectic supermanifolds in \cite{Khudaverdian:2004}.
\end{remark}
\no To use a homogeneous variant of Moser's trick, we also need to integrate time-dependent even vector fields on supermanifolds.
\begin{lemma}\label{l2}
Let $X:\R\ti M\to\sT M$, $X(t,y)=X_t(y)$, be a smooth family of even vector fields on the supermanifold $M$ and let $m\in|M|$.
Then, there is an open submanifold $U\subset M$ around $m\in|U|$, and a one-parameter family of smooth maps $\zF:[0,1]\ti U\to M$, $\zF(t,y)=\zF_t(y)$, which is the flow of the time-dependent vector field $X$, i.e., $\zF_t:U\to M$ are local diffeomorphisms for $t\in[0,1]$, $\zF_0=\id$, and
\be\label{e8}\frac{\xd}{\xd t}\zF_t(y)=X_t\big(\zF_t(y)\big)\,.\ee
If $M$ is additionally a homogeneity supermanifold and $X$ is of weight 0, then $\zF_t$ are morphisms of homogeneity supermanifolds.
\end{lemma}
\begin{proof}
Consider the supermanifold $\tM=\R\ti M$ with the vector field $\wt X(s,y)=\pa_s+X_s(y)$. This vector field is even and nowhere vanishing, $\wt X(s,m)\ne 0$, so it can be integrated to a local flow $\zf_t$ on $\tM$ \cite[Chapter V]{Tuynman:1983} (see also Remark \ref{rw}). This easily follows also from a result by Shander \cite[Theorem 1]{Shander:1980} who proved that any even vector field $Y$ on a supermanifold $N$, with $Y(n)\ne 0$, $n\in|N|$, can be locally written as $Y=\pa_{x^1}$ in a coordinate neighbourhood $(U,(x^i))$, $m\in|U|$.

Choosing $|U|$ sufficiently small, we may assure that the flow of local diffeomorphisms is defined for $0\le t\le 1$. Such a flow $\zf_t$ for the vector field $\wt X$ looks locally like $\zf_t(s,y)=(t+s,\zF(t,y))$, $\zf_0=\id$, and it is easy to see that $\zF(t,y)$ satisfies equation (\ref{e8}). The rest is obvious.
\end{proof}
\no Actually, we can consider the time-dependent vector field $X$ as defined only on a neighbourhood of $\{ 0\}\ti M$ in $\tM$, with no real changes in the proof. Completely analogously to the classical (purely even) situation, one proves the following.

\begin{proposition}\label{prx}
If $\zF_t$ is the local flow of a time-dependent even vector field $X_t$ on a supermanifold $M$, and $K_t$ is a time-dependent tensor field on $M$, then
\be\frac{\xd}{\xd t}\zF_t^\star(K_t)=(\zF_t)^\star\left(\Ll_{X_t}K_t+\dot K_t\right)\,,\ee
where $\zF_t^\star$ denotes the action of the local diffeomorphism on tensor fields (the pullback).
\end{proposition}
\begin{theorem}[\textbf{Homogeneous Darboux Theorem}]\label{thm:Darboux}
 Any  symplectic $\zG$-homogeneity supermanifold $(M,\zw)$ with symplectic form $\zw$ of degree $\zl=(\zs,w)$ admits local Darboux coordinates around any $m\in|M|$ (see \eqref{Darboux}). The coordinates $y_l$ can appear only if $\zw$ is even.
\end{theorem}
\begin{remark}
That the Darboux coordinates have weights in $\zG$ puts some constraints on possible $\zl$ namely $\deg(\bar p_i)+\deg(\bar q^i)=\zl$  and $2\deg(\bar y^j)=\zl$. Note also that our Darboux coordinates need not be allowed if $\nabla_M(m)\ne 0$ and an additional structure determining allowance is involved.
\end{remark}
\begin{proof}
We will adapt Moser's trick for homogeneity supermanifolds.
Let $U$ be a homogeneity superdomain in $M$, with $\zG$-homogeneous coordinates $(x^a)$, and let $m\in|U|$.
According to Theorem \ref{t2}, we can find a coordinate neighbourhood $V\subset U$, such that $m\in|V|$, with homogeneous coordinates
$$(\bar p_1,\dots,\bar p_a,\bar q^1,\dots,\bar q^a,\bar y_1,\dots,\bar y_b)$$
having weights in $\zG$, such that $\zw(m)$ is of the form
\be\label{e6}
\zw(m)=\sum_i \xd\bar p_i(m)\we\xd\bar q^i(m)+\sum_j\ze^j\big(\xd \bar y_j(m)\we\xd\bar y_j(m)\big)\,,
\ee
where $\ze_j=\pm 1$. It is clear that $\deg(\bar p_i)+\deg(\bar q^i)=\zl$  and $2\deg(\bar y^j)=\zl$. Let us consider a two-form on $V$ which reads
\be\label{e7}
\zw_0=\sum_i \xd\bar p_i\we\xd\bar q^i+\sum_j\ze^j(\xd \bar y_j\we\xd\bar y_j)\,.
\ee
Obviously, it is a symplectic form of degree $\zl$ on $V$. Put $\zw_t=(1-t)\zw_0+t\zw$. Of course, $\zw_t$ is a closed form of degree $\zl$. Since $\zw_t(m)=\zw(m)$ is non-degenerate, we can choose $V$ so small that $\zw_t$ is non-degenerate in $V$, thus symplectic, for $t\in[0,1]$.

The form $\zw-\zw_0$ is a closed 2-form of degree $\zl$, so by homogeneous Poincar\'e Lemma \ref{Mt}, for a neighbourhood $V$ of $m$ we have $\zw-\zw_0=\xd\za$ for a one-form $\za$ of degree $\zl$ and vanishing at $m$. Since the two-forms $\zw_t$ are symplectic, there is a time-dependent vector field $X_t$, such that $i_{X_t}\zw_t=-\za$. The vector field $X_t$ is of degree 0, thus even. Since $\za(m)=0$, also $X_t(m)=0$, and we can find a sufficiently small $V$ such that the flow $\zF_t$ of $X_t$ is defined for $0\le t\le 1$.
As the time-dependent vector field is of degree 0, the flow consists of morphisms of $\zG$-homogeneity manifolds.
Like in the classical situation, we have (cf. Proposition \ref{prx})
\be\label{e9}\frac{\xd}{\xd t}\zF_t^*(\zw_t)=(\zF_t)^*\left(\Ll_{X_t}\zw_t+\frac{\xd}{\xd t}\zw_t\right)\,,\ee
where $\Ll_{X_t}=\xd\circ i_{X_t}+i_{X_t}\circ\xd$ is the Lie derivative along $X_t$.
But
$$\Ll_{X_t}\zw_t+\frac{\xd}{\xd t}\zw_t=\xd(i_{X_t}\zw_t)+\zw-\zw_0=-\xd\za+\zw-\zw_0=0\,.$$
Therefore, the form $\zF_t^*(\zw_t)$ is independent on $t$ and
$$\zF_1^*(\zw_1)=\zF_1^*(\zw)=\zF_0^*(\zw_0)=\zw_0\,.$$
This means that $\zw=\big((\zF_1)^{-1}\big)^*(\zw_0)$, and that $\zw$ in coordinates
$$p_i=\bar p_i\circ(\zF_1)^{-1}\,,\ q_i=\bar q^i\circ(\zF_1)^{-1}\,,\ y^j=\bar y^j\circ(\zF_1)^{-1}$$
reads as in (\ref{Darboux}). Since $\zF_1$, coming from the flow of a vector field of degree 0, preserves homogeneity degrees, the new coordinates are also homogeneous with degrees in $\zG$.

\end{proof}
\begin{remark}
For N-manifolds, the Darboux coordinates mentioned in Theorem \ref{thm:Darboux} are also N-homogeneous with weights determining the parity. Hence, we have a homogeneous Darboux Theorem for N-manifolds.
\end{remark}
\begin{corollary}
If $\zt:M\to B$ is a positive homogeneity vector superbundle of rank $(k|l)$ and $m\in|B|$, then as the Darboux coordinates in Theorem \ref{thm:Darboux} we can take allowed coordinates, so that (\ref{Darboux}) is valid on $U\ti \R^{k|l}$, where $U$ is a neighbourhood of $m$ in $B$.
\end{corollary}
\begin{proof}
As (\ref{Darboux}) is valid in a neighbourhood of $m$ in $M$ and homogeneous functions are polynomial, there is a unique (polynomial) extension of (\ref{Darboux}) to fibers of $M$.

\end{proof}
\section{Conclusions}
We have introduced and systematically studied the concept of homogeneity structures on supermanifolds, which capture the essence of grading and generalize several notions of graded manifolds found in the literature. In particular, homogeneity structures provide an elegant and effective tool for working with vector superbundles, offering a perspective that is conceptually more geometric than approaches based on locally free sheaves. In order to retain the full apparatus of differential calculus, we do not start from ringed spaces and sheaves of graded algebras, but instead work with genuine supermanifolds endowed with an additional geometric structure, namely, a weight vector field.

Of course, the corresponding sheaf of algebras of homogeneous functions can be recovered as a subsheaf of the standard structural sheaf of the supermanifold. However, the grading of this sheaf may be an $\R$-grading, which does not induce a direct sum decomposition of the algebra. Within our framework, we introduce and study homogeneous tensor fields, homogeneous submanifolds, and homogeneous distributions, as well as homogeneity Lie supergroups, illustrating all concepts with explicit examples.

We also investigate double homogeneity supermanifolds, with homogeneity vector superbundles as a particular case, as well as tangent and cotangent lifts of homogeneity structures, which produce new homogeneity supermanifolds from existing ones. We provide a natural definition of homogeneous distributions. The main results of the paper, which are of fundamental importance, are the homogeneous Poincar\'e Lemma and the homogeneous Frobenius Theorem. As an application, we study homogeneous symplectic forms and prove a homogeneous version of the symplectic Darboux Theorem.

\medskip\noindent
This paper represents only an initial step in a broader investigation of homogeneity structures in differential geometry, and many directions remain to be explored. In particular, one may introduce homogeneous structures on homogeneity supermanifolds for a much wider class of geometric objects than those considered here, including homogeneous principal bundles, homogeneous connections, homogeneous Riemannian manifolds, and homogeneous contact structures. Another promising direction is the development of a homogeneous analogue of Darboux’s classification of Pfaffian forms. We plan to address some of these topics in future work.

\small{\vskip1cm}

\noindent Katarzyna GRABOWSKA\\
Faculty of Physics, University of Warsaw\\
ul. Pasteura 5, 02-093 Warszawa, Poland
\\Email: konieczn@fuw.edu.pl\\
https://orcid.org/0000-0003-2805-1849\\

\noindent Janusz GRABOWSKI\\ Institute of
Mathematics,  Polish Academy of Sciences\\ ul. \'Sniadeckich 8, 00-656 Warszawa, Poland
\\Email: jagrab@impan.pl\\
https://orcid.org/0000-0001-8715-2370


\begin{thebibliography}{V}
\bibitem{Aizawa:2024} N.~Aizawa, R.~Ito, T.~Tanaka,
\newblock{$\Z_2^2$-graded supersymmetry via superfield on minimal $\Z_2^2$-superspace,}
\newblock{\emph{J. Phys. A} \textbf{57} (2024), 435201, 19pp.}

\bibitem{Benoit:2022} J.~Benoit, A.~Kotov, N.~Poncin, V.~Salnikov,
\newblock{Differential graded Lie groups and their differential graded Lie algebras,}
\newblock{\emph{Transform. Groups} \textbf{27} (2022), 497--523.}

\bibitem{Berezin:1987} F.~A.~Berezin,
\newblock{Introduction to superanalysis,}
\newblock{\emph{Mathematical Physics and Applied Mathematics}, vol. 9,
D.~Reidel Publishing Co., Dordrecht, 1987.}

\bibitem{Bruce:2014} A.~J.~Bruce,
\newblock{On curves and jets of curves on supermanifolds,}
\newblock{\emph{Arch. Math. (Brno)} \textbf{50} (2014), 115--130.}

\bibitem{Bruce:2020} A.~J.~Bruce,
\newblock{$\Z_2\ti\Z_2$-graded supersymmetry: 2-d sigma models,}
\newblock{\emph{J. Phys. A} \textbf{53} (2020), 455201, 25 pp.}

\bibitem{Bruce:2015} A.~J.~Bruce, K.~Grabowska, J.~Grabowski,
\newblock{Graded bundles in the category of Lie groupoids,}
\newblock{\emph{SIGMA Symmetry Integrability Geom. Methods Appl.} \textbf{11} (2015), Paper 090, 25 pp.}

\bibitem{Bruce:2015a} A.~J.~Bruce, K.~Grabowska, J.~Grabowski,
\newblock{Higher order mechanics on graded bundles,}
\newblock{\emph{J. Phys. A} \textbf{48} (2015), 205203, 32 pp.}

\bibitem{Bruce:2016} A.~J.~Bruce, K.~Grabowska, J.~Grabowski,
\newblock{Linear duals of graded bundles and higher analogues of (Lie) algebroids},
\newblock{\emph{J. Geom. Phys.} \textbf{101 }(2016), 71--99.}

\bibitem{Bruce:2017a} A.~J.~Bruce, K.~Grabowska, J.~Grabowski,
\newblock{Introduction to graded bundles,}
\newblock{\emph{Note Mat.} \textbf{37} (2017), suppl. 1, 59--74.}

\bibitem{Bruce:2021} A.~J.~Bruce, J.~Grabowski,
\newblock{Symplectic $\Z^n_2$-manifolds,}
\newblock{\emph{J. Geom. Mech.} \textbf{13} (2021), 285--311.}

\bibitem{Bruce:2025} A.~J.~Bruce, J.~Grabowski,
\newblock{Principal bundles in the category of $\Z_2^n$-manifolds,}
\newblock{\emph{Differ. Geom. Appl.} (2025).} 

\bibitem{Bruce:2021a} A.~J.~Bruce, E.~Ibargu\:engoytia, N. Poncin,
\newblock{Linear $Z^n_2$-manifolds and linear actions,}
\newblock{\emph{SIGMA Symmetry Integrability Geom. Methods Appl.} \textbf{17} (2021), Paper No. 060, 58 pp.}

\bibitem{Bruzzo:1985} U.~Bruzzo, R.~Cianci,
\newblock{Differential equations, Frobenius theorem and local flows on supermanifolds,}
\newblock{\emph{J. Phys. A.} \textbf{18} (1985), 417--423.}

\bibitem{Carmeli:2011} C.~Carmeli, L.~Caston, R.~Fioresi,
\newblock{Mathematical foundations of supersymmetry,}
\newblock{\emph{EMS Ser. Lect. Math.},
European Mathematical Society (EMS), Z\"urich, 2011.}

\bibitem{Covolo:2016a}
T.~Covolo, J.~Grabowski, N.~Poncin,
\newblock{The category of {$\mathbb{Z}_2^n$}-supermanifolds,}
\newblock{\emph{J.~Math. Phys.} \textbf{57} (2016), 073503, 16~pages.}

\bibitem{Covolo:2016b}
T.~Covolo, J.~Grabowski, N.~Poncin,
\newblock{Splitting theorem for {$\mathbb{Z}_2^n$}-supermanifolds,}
\newblock{\emph{J.~Geom. Phys.} \textbf{110} (2016), 393--401.}


\bibitem{Deligne:1999} P.~Deligne, J.~W.~Morgan,
\newblock{Notes on supersymmetry (following Joseph Bernstein),}
\newblock{\emph{Quantum fields and strings: a course for mathematicians}, Vol. 1 (Princeton, NJ, 1996/1997), 41--97, Amer. Math. Soc., Providence, RI, 1999.}

\bibitem{Fairon:2017}   M.~Fairon,
\newblock{Introduction to graded geometry,}
\newblock{\emph{European J. Math.} \textbf{3} (2017), 208--222.}

\bibitem{Garnier:2013} S~Garnier, T.~Wurzbacher,
\newblock{Integration of vector fields on smooth and holomorphic supermanifolds,
\newblock{\emph{Doc. Math.} \textbf{18} (2013), 519--545.}}

\bibitem{Gawedzki:1977} K.~Gaw\c edzki,
\newblock{Supersymmetries-mathematics of supergeometry.}
\newblock{\emph{Ann. Inst. H. Poincar\'e, A} \textbf{27} (1977), 335--366.}

\bibitem{Grabowska:2021} K.~Grabowska, J.~Grabowski and Z.~Ravanpak.
\newblock{VB-structures and generalizations.}
\newblock{\emph{Ann. Global Anal. Geom.} \textbf{62} (2022), 235--284.}

\bibitem{Grabowski:2013} J.~Grabowski,
\newblock{Graded contact manifolds and contact Courant algebroids}
\newblock{\emph{J. Geom. Phys. } \textbf{68} (2013), 27--58.}

\bibitem{Grabowski:2009} J.~Grabowski, M.~Rotkiewicz,
\newblock{Higher vector bundles and multi-graded symplectic manifolds,}
\newblock{\emph{J. Geom. Phys.} \textbf{59} (2009), 1285--1305}

\bibitem{Grabowski:2012} J.~Grabowski, M.~Rotkiewicz,
\newblock{Graded bundles and homogeneity structures,}
\newblock{\emph{J. Geom. Phys.} \textbf{62} (2012), 21--36.}

\bibitem{Grabowski:1995} J.~Grabowski, P.~Urba\'nski,
\newblock{Tangent lifts of Poisson and related structures,}
\newblock{\emph{J. Phys. A} \textbf{28} (1995), 6743--6777.}

\bibitem{Grabowski:1997} J.~Grabowski, P.~Urba\'nski,	
\newblock{Lie algebroids and Poisson-Nijenhuis structures,}
\newblock{Quantizations, deformations and coherent states (Białowieża, 1996).}
\newblock{\emph{Rep. Math. Phys.} \textbf{40} (1997), 195--208.}

\bibitem{Grabowski:1999} J.~Grabowski, P.~Urba\'nski,
\newblock{Algebroids—general differential calculi on vector bundles,}
\newblock{\emph{J. Geom. Phys.} \textbf{31} (1999), 111--141.}

\bibitem{Kontsevich:2003} M.~Kontsevich
\newblock{Deformation quantization of Poisson manifolds,}
\newblock{\emph{Lett. Math. Phys.} \textbf{66} (2003), 157--216.}

\bibitem{Kostant:1977} B.~Kostant,
\newblock{Graded manifolds, graded Lie theory, and prequantization,}
\newblock{{\sl Differential geometrical methods in mathematical physics} (Proc. Sympos., Univ. Bonn, Bonn, 1975), pp. 177--306. \emph{Lecture Notes in Math.}, Vol. \textbf{570}, Springer, Berlin, 1977.}

\bibitem{Khudaverdian:2004} H.~M.~Khudaverdian,
\newblock{Semidensities on odd symplectic supermanifolds,}
\newblock{\emph{Comm. Math. Phys.} \textbf{247} (2004), 353--390.}

\bibitem{Kotov:2024} A.~Kotov, V.~Salnikov,
\newblock{The category of $\Z$-graded manifolds: what happens if you do not stay positive,}
\newblock{\emph{Differential Geom. Appl.} \textbf{93} (2024), Paper No. 102109, 25 pp.}

\bibitem{Jiang:2023} S.~Jiang,
\newblock{Monoidally graded manifolds,}
\newblock{\emph{J. Geom. Phys.} \textbf{183} (2023), Paper No. 104701, 13 pp.}

\bibitem{Jozwikowski:2016} M.~J\'o\'zwikowski and M.~Rotkiewicz,
\newblock{A note on actions of some monoids,}
\newblock{\emph{Differential Geom. Appl.} \textbf{47} (2016), 212--245.}

\bibitem{Leites:1980} D.~A.~Leites,
\newblock{Introduction to the theory of supermanifolds,}
\newblock{\emph{Russ. Math. Surv.} \textbf{35}  (1980), 1--64.}

\bibitem{Lundell:1992} A.~T.~Lundell,
\newblock{A short proof of the Frobenius theorem,}
\newblock{\emph{Proc. Amer. Math. Soc.} \textbf{116} (1992), 1131--1133.}

\bibitem{Mehta:2006} R.~A.~Mehta,
\newblock{Supergroupoids, Double Structures, and Equivariant Cohomology,}
\newblock{PhD thesis, University of California, Berkeley, ProQuest LLC, Ann Arbor (2006),  133 pages,
\emph{arXiv:math.DG/0605356}.}

\bibitem{Rogers:2007} A.~Rogers,
\newblock{Supermanifolds. Theory and Applications. \emph{World Scientific Publishing Co. Pte. Ltd.}, Hackensack, NJ, 2007.}

\bibitem{Roytenberg:2002} D.~Roytenberg,
\newblock{On the structure of graded symplectic supermanifolds and Courant algebroids,}
\newblock{In {\sl Quantization, Poisson brackets and beyond (Manchester, 2001)}, 169--185, \emph{Contemp. Math.} \textbf{315}, Amer. Math. Soc., Providence, RI, 2002.}

\bibitem{Salnikov:2021} V.~Salnikov, A.~Hamdouni, D.~Loziienko,
\newblock{Generalized and graded geometry for mechanics: a comprehensive introduction,}
\newblock{\emph{Math. Mech. Complex Syst.} \textbf{9} (2021), 59--75.}

\bibitem{Schwarz:1990}\vskip-.2cm A.~Schwarz,
\newblock{Symplectic, contact and superconformal geometry, membranes and strings,}
\newblock{{\sl Supermembranes and physics in $2+1$ dimensions} (Trieste, 1989), 332--351, World Sci. Publ.,
River Edge, NJ, 1990.}

\bibitem{Schwarz:1993} A.~S.~Schwarz,
\newblock{Geometry of Batalin-Vilkovisky quantization,}
\newblock{\emph{Comm. Math. Phys.} \textbf{155} (1993), 249--260.}

\bibitem{Schwarz:1996} A.~S.~Schwarz,
\newblock{Superanalogs of symplectic and contact geometry and their applications to quantum field theory,}
\newblock{in \emph{Topics in statistical and theoretical physics}, 203--218, Amer. Math. Soc. Transl. Ser. 2, \textbf{177}, \emph{Adv. Math. Sci.} \textbf{32}, {Amer. Math. Soc., Providence, RI}, 1996.}

\bibitem{Severa:2005} P.~{\v{S}}evera,
\newblock{Some title containing the words "homotopy" and "symplectic", e.g. this one,}
\newblock{\emph{Travaux math\'ematiques}, Univ. Luxemb., \textbf{16} (2005), 121--137.}

\bibitem{Shander:1980} V.~N.~Shander,
\newblock{Vector fields and differential equations on supermanifolds,}
\newblock{\emph{Funct. Anal. Its Appl.} \textbf{14} (1980), 160--162.}

\bibitem{Shander:1983} V.~N.~Shander,
\newblock{Analogues of the Frobenius and Darboux theorems for supermanifolds (Russian),}
\newblock{\emph{C. R. Acad. Bulgare Sci.} \textbf{36} (1983), 309--312.}

\bibitem{Tulczyjew:1974} W.~M.~Tulczyjew,
\newblock{Hamiltonian Systems, Lagrangian systems and the Legendre transformation,}
\newblock{\emph{Symp. Math.} \textbf{14}, Roma (1974), 247--258.}

\bibitem{Tuynman:1983} G.~M.~Tuynman,
\newblock{Supermanifolds and supergroups,}
\newblock{\emph{Mathematics and its Applications}, vol. \textbf{570},
Kluwer Academic Publishers, Dordrecht, 2004.}

\bibitem{Voronov:2019} Th.~Th.~Voronov,
\newblock{Graded geometry, Q-manifolds, and microformal geometry,}
\newblock{\emph{Fortschr. Phys.} 67 (2019), 1910023, 29 pp.}

\bibitem{Voronov:2002} Th.~Th.~Voronov,
\newblock{Graded manifolds and {D}rinfeld doubles for {L}ie bialgebroids}.
\newblock{In: Quantization, Poisson Brackets and Beyond,
\emph{Contemp. Math.} \textbf{315}, 131--168,}
\newblock{Amer. Math. Soc., Providence, RI, 2002.}

\bibitem{Smolka:2025} R.~\v{S}molka and J.~Vysok\'y,
\newblock{Threefold nature of graded vector bundles,}
\newblock{\emph{J. Geom. Phys.} 2025, 105557.}

\bibitem{Vysoky:2022} J.~Vysok\'y,
\newblock{Graded generalized geometry,}
\newblock{\emph{J. Geom. Phys.} \textbf{182} (2022), Paper No. 104683, 27 pp.}

\bibitem{Vysoky:2022a} J.~Vysok\'y,
\newblock{Global theory of graded manifolds,}
\newblock{\emph{Rev. Math. Phys.} \textbf{34} (2022), Paper No. 2250035, 197 pp.}

\bibitem{Vysoky:2024} J.~Vysok\'y,
\newblock{Graded jet geometry,}
\newblock{\emph{J. Geom. Phys.} \textbf{203} (2024), Paper No. 105250, 31 pp.}

\bibitem{Yano:1973} K.~Yano, S.~Ishihara,
\newblock{Tangent and Cotangent Bundles,}
\newblock{Marcel Dekker, Inc., 1973.}

\end{thebibliography}
\end{document}